\documentclass[leqno,CJK]{siamltex704}
\usepackage{amssymb,amsmath,graphicx,amscd,mathrsfs}
\usepackage{color,xcolor,amsmath}
\usepackage{amsmath}
\usepackage{graphicx}
\usepackage{mathrsfs}
\usepackage{float}
\usepackage{amsfonts,amssymb}
\usepackage{dsfont}
\usepackage{pifont}
\usepackage{hyperref}
\usepackage{multirow}
\usepackage{placeins}
\usepackage{geometry}
\usepackage[numbers,sort&compress]{natbib}
\numberwithin{equation}{section}
\def\3bar{{|\hspace{-.02in}|\hspace{-.02in}|}}
\def\E{{\mathcal{E}}}
\def\T{{\mathcal{T}}}

\def\dQ{{\mathbb{Q}}}

\def\b0{\boldsymbol{0}}
\def\sumT{\sum_{T\in\mathcal{T}_h}}     

\def\bn{{\mathbf{n}}}
\def\bb{{\mathbf{b}}}
\def\bd{{\mathbf{d}}}
\def\bc{{\mathbf{c}}}

\def\bf{{\mathbf{f}}}

\def\m{\max{\{k+1,4\}}}

 \newcommand{\WL}{\Delta _w}

 \newcommand{\Real}{\mathbb{R}}

 \newcommand{\norm}[1]{\left\Vert#1\right\Vert}

 \newcommand{\trb}[1]{|\!|\!|#1|\!|\!|}

 \newcommand{\G}[1]{\nabla  #1}

\newcommand{\OP}{\overline{\partial}}

\allowdisplaybreaks 

\begin{document}

\title{A stabilizer free weak Galerkin method with implicit $\theta$-schemes for fourth order parabolic problems}

\author{
  Shanshan Gu
  \and
Qilong Zhai
}
\maketitle
\begin{abstract}
In this paper, we combine the stabilizer free weak Galerkin (SFWG) method and the implicit $\theta$-schemes in time for $\theta\in [\frac{1}{2},1]$ to solve the fourth-order parabolic problem. In particular, when $\theta =1$, the full-discrete scheme is first-order backward Euler and the scheme is second-order Crank Nicolson scheme if $\theta =\frac{1}{2}$. Next, we analyze the well-posedness of the schemes and deduce the optimal convergence orders of the error in the $H^2$ and $L^2$ norms. Finally, numerical examples confirm the theoretical results.
\end{abstract}



\section{Introduction}

 In this paper, we consider the fourth order parabolic equation as follows:
Find $u$ such that
\begin{equation}\label{problem-eq}
\left\{
\begin{array}{rcl}
u_t + \Delta ^2 u ~~= &f,\qquad &x\in\Omega,~t\in J,\\
u ~~=&0,\qquad &x\in\partial\Omega,~t\in J,\\
\frac{\partial u}{\partial \bn} ~~=&0,\qquad &x\in\partial\Omega,~t\in J,\\
u(0,\cdot) ~~=&\psi,\qquad &x\in\Omega,
\end{array}
\right.
\end{equation}
where $\Omega \subset \Real^d$ $(d=2,3)$ is a polygon region and $J=(0,\overline{t}]$ with $\overline{t}>0$ is a time interval.

By the integration by parts, we can get the following variational formulation: seek $u(t,\cdot)\in H^1[0,\overline{t};H^2_0(\Omega)]$ such that $u(0,\cdot)=\psi$ and
\begin{align}\label{var-form}
  (u_t,v)+(\Delta u,\Delta v)=(f,v),\qquad \forall v\in H^2_0(\Omega),
\end{align}
where we use the standard definition of the Sobolev space for $H^s(\Omega)$. The inner products, norms and seminorms in $H^s(\Omega)$ $(s\geq 0)$ are respectively denoted by $(\cdot ,\cdot)_s$, $\norm{\cdot}_s$ and $|\cdot|_s$. When $s=0$, the subscript $s$ can be omitted. In particularly, we know $H^2_0(\Omega)=\{v\in H^2(\Omega):v|_{\partial \Omega}=\frac{\partial v}{\partial \bn}|_{\partial\Omega}=0\}$.

For fourth order parabolic equations, there are various methods to solve it, such as the conforming finite element methods\cite{MR1742748,MR3109557}, the mixed finite element methods\cite{li_optimal_2006,liu_coupling_2013,li_mixed_2005,liu_mixed_2018} and discontinuous Galerkin finite element methods\cite{MR4382455}.

Weak Galerkin (WG) finite element method, as a discontinuous Galerkin finite element method, is proposed in \cite{PossionWG} to solve the second order elliptic problem. The method introduces the definitions of weak functions and weak operators to build a new numerical formulation. Recently, this method has been developed rapidly. On the one hand, WG method has been applied to the Stokes equation \cite{StokesWG}, the biharmonic equation \cite{BiharmonicWG,BiharmonicWGReOrder}, the Brinkman equation \cite{BrinkmanWG} and so on. On the other hand, the method has been further improved. Modified weak Galerkin finite element method, using weak function $\{v_0,\{v_0\}\}$ instead of $\{v_0,v_b\}$, extends the WG method and has been used to solve the Poisson equation \cite{PossionMWG,hussain_study_2022}, the biharmonic equation \cite{BiharmonicMWG}, the Brinkman equation \cite{sun_modied_nodate}, the Stokes equation \cite{mu_modified_2015}, the parabolic equation \cite{gao_modified_2014}, etc. By raising the degree of the polynomial in the space of the range of the weak operator, stabilizer free weak Galerkin (SFWG) finite element method \cite{PossionSFWG} removes the stabilizer in the WG format and achieves the same error convergence orders as the WG method. At present, the SFWG method has been utilized to solve the second order elliptic problem \cite{PossionSFWG,ye_stabilizer_2021-1,ye_stabilizer_2021,ye_new_2021}, the Stokes equation \cite{StokesSFWG}, the biharmonic equation \cite{BiharmonicSFWG,zhu_stabilizer-free_2022}, the parabolic equation \cite{MR4217892}, etc. Combining the MWG method with the SFWG method, conforming discontinuous Galerkin (CDG) finite element method is proposed in \cite{PossionCDG1,PossionCDG2,ye_conforming_2021,ye_conforming_2019,ye_weak_2022}, which simply the numerical scheme and reduce the degree of freedom of the weak function space.

Like other equations, the fourth order parabolic equation is solved by WG method in \cite{chai_weak_2019}. The aim of this paper is to use the SFWG method to solve the fourth order parabolic equation. In addition, we use the implicit $\theta$-schemes in time for $\theta\in [\frac{1}{2},1]$ to build the full discrete numerical scheme, which is utilized to solve the second order parabolic problem \cite{qi_weak_2020} and can reach $O(\tau ^2)$ about the time step $\tau$ when $\theta =\frac{1}{2}$.

An outline of the paper is organized as follows. In Section 2, we present the SFWG schemes, including the semi-discrete numerical scheme and the full-discrete numerical scheme. In Section 3, we introduce the norms of the weak finite element space and analyze the well-posedness of the two formulations. Next, the error estimates are obtained in Section 4. In Section 5, we use numerical examples to verify the rationality of the order of theoretical convergence. In the final Section, we make some conclusions.
\section{A standard discretization of weak Galerkin scheme}
In this section, we define some notations and build the stabilizer free WG scheme for the fourth order parabolic equation.


Let $\T_h$ be a partition of the domain $\Omega$ which satisfies the regular assumptions in \cite{Wang2014a}, $\E_h$ be all edges in $\T_h$ and
$\E_h^0$ be all interior edges $\E_h\backslash \partial\Omega$. For
each element $T\in\T_h$, $h_T$ denotes the diameter of $T$ and
$h=\max_{T\in\T_h} h_T$ represents the mesh size.

Now, we define the set of normal directions on $\E_h$ as follows
\begin{eqnarray*}
  \mathcal{D} _h=\{\bn _e:\bn_e\text{ is unit and normal to }e,~e\in\E_h\}.
\end{eqnarray*}

Then, for a given integer $k\ge 2$, we can define the WG finite
element spaces
\begin{eqnarray*}
V_h=\{v=\{v_0,v_b,v_n\bn_e \}:v_0|_T\in P_k(T), v_b|_e\in P_{k}(e), v_n|_e\in P_{k-1}(e), T\in\T_h, e\in\E_h\},
\end{eqnarray*}
and
\begin{eqnarray*}
  V_h^0=\{v=\{v_0,v_b,v_n\bn_e\}\in V_h:v_b|_e=0, v_n|_e=0, e\subset \partial\Omega\}.
\end{eqnarray*}

For each weak function $v\in V_h+H^2(\Omega)$, we can define its weak laplace $\Delta_w v$ as follows.
\begin{definition}\cite{BiharmonicSFWG}
For each $v\in V_h+H^2(\Omega)$, $\Delta_w v|_T$ is the unique polynomial in
$P_{j}(T)$ satisfying
\begin{eqnarray}\label{def-wlaplace}
  (\Delta_w v,\varphi)_{T}=(v_0,\Delta\varphi)_{T}-\langle v_b,\nabla\varphi\cdot\bn\rangle_{\partial T}+
  \langle v_n\bn_e\cdot\bn,\varphi\rangle_{\partial T},\quad
  \forall \varphi\in P_{j}(T),
\end{eqnarray}
where $\bn$ denotes the outward unit normal vector.
\end{definition}

To facilitate the analysis, we introduce some projection operators. $Q_0$ denotes the $L^2$ projection from $L^2(T)$
onto $P_k(T)$, $Q_b$ denotes the $L^2$ projection from $L^2(e)$ onto
$P_{k}(e)$, $Q_n$ denotes the $L^2$ projection from $L^2(e)$ onto
$P_{k-1}(e)$ and $\dQ_h$ denotes the $L^2$ projection from
$L^2(T)$ onto $P_{j}(T)$. Further, we define $Q_h=\{Q_0,Q_b,Q_n\}$ to represent the projection from $H^2(\Omega)$ onto $V_h$, that is, for exact solution $u$,
\begin{align*}
  Q_hu=\{Q_0u,Q_bu,Q_n(\G u\cdot\bn_e)\}.
\end{align*}

Substitute the weak laplace operator for the strong laplace operator in (\ref{var-form}), and we build the following semi-discrete numerical scheme: find $u_h\in H^1(0,\overline{t};V_h^0)$ such that $u_h(0,\cdot)=Q_h\psi$ and
\begin{align}\label{semi-scheme}
  (u_{h,t},v_0)+(\WL u_h,\WL v)=(f,v_0),\qquad \forall v\in V_h^0.
\end{align}

Using the time step $\tau$ to divide the time interval $(0,\overline{t}]$ into subintervals and defining $t_n=n\tau$, we propose the full-discrete numerical scheme: seek $U^n\in V_h^0$, $(n=1,2,\cdots,N)$, such that $U^0=Q_h\psi$ and
\begin{align}\label{full-scheme}
  &(\overline\partial U^n,v_0)+(\WL (\theta U^n+(1-\theta)U^{n-1}),\WL v)\\ \nonumber
  =&(\theta f(t_n)+(1-\theta)f(t_{n-1}),v_0),\qquad \forall v\in V_h^0,
\end{align}
where $\theta\in [\frac{1}{2},1]$, $\overline{\partial}U^n:=\frac{U^n-U^{n-1}}{\tau}$ and $f(t_n):=f(t_n,x)$. When $\theta =1$, it is well-known backward Euler scheme, and the (\ref{full-scheme}) is Crank-Nicolson (CN) scheme if $\theta =\frac{1}{2}$.

\section{Well-posedness}

In this section, we shall introduce the norms of $V_h^0$ and verify the well-posedness of the semi-discrete scheme (\ref{semi-scheme}) and the full-discrete scheme (\ref{full-scheme}).

For any $v\in V_h+H^2(\Omega)$, we can define
\begin{align}
  \label{trb}\trb v ^2&=(\WL v,\WL v),\\
  \label{2h}\norm{v}^2_{2,h}&=\sumT (\norm{\Delta v_0}_T^2+h_T^{-3}\norm{v_0-v_b}^2_{\partial T}+h_T^{-1}\norm{(\nabla v_0-v_n\bn _e)\cdot\bn}^2_{\partial T}).
\end{align}

\begin{lemma}\label{lemma-norm-equ}
  \cite{BiharmonicSFWG}For any $v\in V_h$, there exists two positive constants such that
  \begin{align}
    C_1\norm{v}_{2,h}\leq\trb v\leq C_2\norm{v}_{2,h}.
  \end{align}
\end{lemma}

\begin{lemma}\label{lemma-2h-norm}
  $\norm{\cdot}_{2,h}$ is a norm in $V_h^0$.
\end{lemma}
\begin{proof}
  It is enough to show the positive property of $\norm{\cdot}_{2,h}$. Suppose $v\in V_h^0$ and $\norm{v}_{2,h}=0$, from the definition of $\norm{\cdot}_{2,h}$, we have
  \begin{align*}
    \Delta v_0|_{T}=0,~ (v_0-v_b)|_{\partial T}=0,~ (\nabla v_0-v_n\bn _e)\cdot\bn|_{\partial T}=0,\qquad\forall T\in\T_h.
  \end{align*}

  Due to
  \begin{align*}
    (\nabla v_0-v_n\bn _e)\cdot\bn = \pm (\nabla v_0\cdot \bn_e -v_n),
  \end{align*}
  we know ($\nabla v_0\cdot \bn_e -v_n)|_{\partial T}=0,~\forall T\in\T_h$. It implies $\nabla v_0\cdot\bn_e = v_n$ on any edge.

  We claim that $\nabla v_0=0,~\forall T\in \T_h$ holds true. For the purpose, we use the Gauss formula to get
  \begin{align*}
    \norm{\nabla v_0}^2_T=(\nabla v_0,\nabla v_0)_T=-(\Delta v_0,v_0)_T +\langle \nabla v_0\cdot\bn,v_0 \rangle _{\partial T}=\langle \nabla v_0\cdot\bn,v_0 \rangle _{\partial T}.
  \end{align*}

  To sum over all $T$, we have
  \begin{align*}
    \sumT \norm{\nabla v_0}^2_T=\sumT\langle \nabla v_0\cdot\bn,v_0 \rangle _{\partial T}
  \end{align*}

  For $e\in\E_h^0$, let $T_1$ and $T_2$ be two elements sharing $e$, $v_0^1,v_0^2$ be the values of $v$ on $T_1$ and $T_2$, and $\bn_1,\bn_2$ be the unit outward normal vectors of $T_1,T_2$ on $e$.
  \begin{align*}
    \langle \nabla v_0^1\cdot\bn_1,v_0^1 \rangle _e+\langle \nabla v_0^2\cdot\bn_2,v_0^2 \rangle _e&=\pm \langle v_n,v_0^1-v_0^2 \rangle _e\\
    &=\pm \langle v_n,v_b-v_b \rangle _e\\
    &=0
  \end{align*}

  Combining $\nabla v_0\cdot\bn_e=v_n=0$ on any boundary edge, we obtain
  \begin{align}
    \sumT \norm{\nabla v_0}^2_T=0,
  \end{align}
  which implies $\nabla v_0=0$ on any element $T$. Using $v_0|_e=v_b$, $\forall e\in \E_h$ and $v\in V_h^0$, we have $v=0$ on $\Omega$.
\end{proof}

From Lemma \ref{lemma-norm-equ} and \ref{lemma-2h-norm}, we find $\trb\cdot$ is also the norm in $V_h^0$.

\begin{lemma}\label{lemma-boundness}
  There exists two positive constants which make the following inequalities hold true for any $w,v\in V_h$:
  \begin{align}
    |a(w,v)|&\leq\trb w\trb v,\\
    \trb v ^2&\leq a(v,v).
  \end{align}
\end{lemma}

\begin{theorem}
  For any $h\in (0,\overline{h}]$, there exists a solution $u_h(t)$ such that the semi-discrete numerical scheme (\ref{semi-scheme}).
\end{theorem}
\begin{proof}
  Since the finite dimensional space $V_h^0$ is the set of weak functions composed of three parts, we suppose the basis of $V_h^0$ is $\{\varphi _1,0,0\}$,$\{\varphi _2,0,0\}$,$\cdots$,$\{\varphi _M,0,0\}$,$\{0,\psi _1,0\}$,$\{0,\psi _2,0\}$,$\cdots$,$\{0,\psi _N,0\}$,$\{0,0,\varsigma _1\}$,$\{0,0,\varsigma _2\}$,\\$\cdots$,$\{0,0,\varsigma _P\}$ and $u_h(t)=\{\sum _{i=1}^Mb_i(t)\varphi _i,\sum _{i=1}^Nd_i(t)\psi _i,\sum _{i=1}^Pc_i(t)\varsigma _i\}$, where the dimension of $V_h^0$ is $M+N+P$. From the semi-discrete formulation (\ref{semi-scheme}), we have
  \begin{align*}
    \sum _{j=1}^M b_j^\prime  (t)(\varphi _j(x),\varphi _i(x))+\Big(\WL\{\sum _{j=1}^Mb_j(t)\varphi _j,\sum _{j=1}^Nd_j(t)\psi _j,\sum _{j=1}^Pc_j(t)\varsigma _j\},\WL\{\varphi _i,0,0\}\Big)&=(f,\varphi _i),&\ i=1,2,\cdots,M,\\
    \Big(\WL\{\sum _{j=1}^Mb_j(t)\varphi _j,\sum _{j=1}^Nd_j(t)\psi _j,\sum _{j=1}^Pc_j(t)\varsigma _j\},\WL\{0,\psi _i,0\}\Big)&=0,&\ i=1,2,\cdots,N,\\
    \Big(\WL\{\sum _{j=1}^Mb_j(t)\varphi _j,\sum _{j=1}^Nd_j(t)\psi _j,\sum _{j=1}^Pc_j(t)\varsigma _j\},\WL\{0,0,\varsigma _i\}\Big)&=0,&\ i=1,2,\cdots,P.
  \end{align*}
  Let $\bb(t)=[b_1(t),b_2(t),\cdots,b_M(t)]^T$, $\bd(t)=[d_1(t),d_2(t),\cdots,d_N(t)]^T$, $\bc(t)=[c_1(t),c_2(t),\cdots,c_P(t)]^T$, $Q_h\psi =\{\sum _{j=1}^M\hat b_j\varphi _j,\sum _{j=1}^N\hat d_j\psi _j,\sum _{j=1}^P\hat c_j\varsigma _j\}$, $\hat{\bb}=[\hat b _1,\hat b _2,\cdots,\hat b _M]^T$, $\hat{\bd}=[\hat d_1,\hat d_2,\cdots,\hat d_N]^T$, $\hat{\bc}=[\hat c_1,\hat c_2,\cdots, \hat c_P]^T$ and
  \begin{align*}
      C=\left[
      \begin{array}{ccc}
        (\varphi _1,\varphi _1) & \cdots & (\varphi _M,\varphi _1)\\
        \vdots & \ddots & \vdots \\
        (\varphi _1,\varphi _M) & \cdots & (\varphi _M,\varphi _M)
      \end{array}
    \right],
    A=\left[
      \begin{array}{ccc}
        A_{00} & A_{0b} & A_{0g}\\
        A_{b0} & A_{bb} & A_{bg}\\
        A_{g0} & A_{gb} & A_{gg}
      \end{array}
    \right],
    \bf(t)=\left[
      \begin{array}{c}
        (f,\varphi _1)\\
        \vdots \\
        (f,\varphi _M)
      \end{array}
    \right],
  \end{align*}
  where
  \begin{align*}
    A_{00}&=\left[(\WL \{\varphi _j,0,0\},\WL\{\varphi _i,0,0\})_{\T_h}\right]_{i,j=1}^M,\\
    A_{0b}&=\left[(\WL \{0,\psi _j,0\},\WL\{\varphi _i,0,0\})_{\T_h}\right]_{j=1,i=1}^{N,M},\\
    A_{0g}&=\left[(\WL \{0,0,\varsigma _j\},\WL\{\varphi _i,0,0\})_{\T_h}\right]_{j=1,i=1}^{P,M},\\
    A_{bb}&=\left[(\WL \{0,\psi _j,0\},\WL\{0,\psi _i,0\})_{\T_h}\right]_{j,i=1}^{N},\\
    A_{bg}&=\left[(\WL \{0,0,\varsigma _j\},\WL\{0,\psi _i,0\})_{\T_h}\right]_{j=1,i=1}^{P,N},\\
    A_{gg}&=\left[(\WL \{0,0,\varsigma _j\},\WL\{0,0,\varsigma _i\})_{\T_h}\right]_{j,i=1}^{P},\\
    A_{b0}&=A_{0b}^T,\\
    A_{g0}&=A_{0g}^T,\\
    A_{gb}&=A_{bg}^T.
  \end{align*}
  
  Then, finding the solution of semi-discrete numerical scheme translates to solving the following problem: finding $\bb(t)$, $\bd(t)$, $\bc(t)$ such that $\bb(0)=\hat\bb$, $\bd(0)=\hat\bd$, $\bc(0)=\hat\bc$ and
  \begin{align}\label{problem-mat}
    \left\{
      \begin{array}{rcl}
        C\bb ^\prime (t)+A_{00}\bb (t)+\left[A_{0b}\ A_{0g}\right]\left[\begin{array}{c}
          \bd(t)\\
          \bc(t)
        \end{array}\right] &=&\bf(t),\\
        \\
        \left[\begin{array}{c}
          A_{b0}\\
          A_{g0}
        \end{array}\right]\bb(t)+\left[\begin{array}{cc}
          A_{bb} & A_{bg}\\
          A_{gb} & A_{gg}
        \end{array}\right]\left[\begin{array}{c}\bd(t) \\ \bc(t) \end{array}\right]&=&\mathbf{0}.
      \end{array}
    \right.
  \end{align}
  for all $t\in(0,T]$.
  
  For any $\boldsymbol{\xi} =[\xi _1,\xi _2,\cdots,\xi _M]^T\neq \mathbf{0}$, we have $\boldsymbol{\xi}^TC\boldsymbol{\xi}=(\sum_{j=1}^M\xi _j\varphi _j,\sum_{i=1}^M\xi _i\varphi _i)>0$, which implies $C$ is a symmetric positive definite matrix and further is an invertible matrix.
  
  For any $\left[\begin{array}{c}
    \boldsymbol{\xi}\\
    \boldsymbol{\eta}
  \end{array}\right]\in \mathbb{R}^{N+P}\setminus \{\mathbf 0\}$, ($\boldsymbol{\xi}\in\mathbb R^N$ and $\boldsymbol{\eta}\in\mathbb R^P$), we have
  \begin{align*}
    &\left[\begin{array}{c}
      \boldsymbol{\xi}\\
      \boldsymbol{\eta}
    \end{array}\right]^T
    \left[\begin{array}{cc}
      A_{bb} & A_{bg}\\
      A_{gb} & A_{gg}
    \end{array}\right]
    \left[\begin{array}{c}
      \boldsymbol{\xi}\\
      \boldsymbol{\eta}
    \end{array}\right]\\
    =&\left[\boldsymbol{\xi}^TA_{bb}\boldsymbol\xi +\boldsymbol\eta ^T A_{gb}\boldsymbol\xi +\boldsymbol\xi ^TA_{bg}\boldsymbol\eta +\boldsymbol\eta ^TA_{gg}\boldsymbol\eta\right]\\
    =&\Big(\WL (\sum_{j=1}^N\xi _j\{0,\psi _j,0\}),\WL (\sum _{j=1}^N\xi _j\{0,\psi _j,0\})\Big)_{\T_h}+\Big(\WL (\sum _{j=1}^N\xi _j\{0,\psi _j,0\}),\WL (\sum_{j=1}^P\eta _j\{0,0,\varsigma _j\})\Big)_{\T_h}\\
    &+\Big(\WL (\sum _{j=1}^P\eta _j\{0,0,\varsigma _j\}),\WL (\sum_{j=1}^N\xi _j\{0,\psi _j,0\})\Big)_{\T_h}+\Big(\WL (\sum_{j=1}^P\eta _j\{0,0,\varsigma _j\}),\WL (\sum _{j=1}^P\eta _j\{0,0,\varsigma _j\})\Big)_{\T_h}\\
    =&\trb{\sum_{j=1}^N\xi _j\{0,\psi _j,0\}+\sum _{j=1}^P\eta _j\{0,0,\varsigma _j\}}^2.
  \end{align*}
  Since $\{0,\psi _1,0\}$,$\{0,\psi _2,0\}$,$\cdots$,$\{0,\psi _N,0\}$,$\{0,0,\varsigma _1\}$,$\{0,0,\varsigma _2\}$,$\cdots$,$\{0,0,\varsigma _P\}$ are linearly independent, we know $$\sum_{j=1}^N\xi _j\{0,\psi _j,0\}+\sum _{j=1}^P\eta _j\{0,0,\varsigma _j\}\neq 0.$$ It follows that $\trb{\sum_{j=1}^N\xi _j\{0,\psi _j,0\}+\sum _{j=1}^P\eta _j\{0,0,\varsigma _j\}}\neq 0$ from $\trb{\cdot}$ is the norm in $V_h^0$, which implies 
  \begin{align*}
    &\left[\begin{array}{c}
      \boldsymbol{\xi}\\
      \boldsymbol{\eta}
    \end{array}\right]^T
    \left[\begin{array}{cc}
      A_{bb} & A_{bg}\\
      A_{gb} & A_{gg}
    \end{array}\right]
    \left[\begin{array}{c}
      \boldsymbol{\xi}\\
      \boldsymbol{\eta}
    \end{array}\right]>0.
  \end{align*}
  Therefore, the matrix
  \begin{align*}
    \left[\begin{array}{cc}
      A_{bb} & A_{bg}\\
      A_{gb} & A_{gg}
    \end{array}\right]
  \end{align*} is invertible.

  By the second equation of (\ref{problem-mat}), we get
  \begin{align}\label{dc-exp}
    \left[\begin{array}{c}\bd(t) \\ \bc(t) \end{array}\right]=-
    \left[\begin{array}{cc}
      A_{bb} & A_{bg}\\
      A_{gb} & A_{gg}
    \end{array}\right]^{-1}
    \left[\begin{array}{c}
      A_{b0}\\
      A_{g0}
    \end{array}\right]\bb(t).
  \end{align}

  Substituting (\ref{dc-exp}) into the first equation of (\ref{problem-mat}), (\ref{problem-mat}) is equivalent to the following problem: finding $\bb(t)$ such that 
  \begin{align}\label{problem-mat2}
    \left\{
      \begin{array}{rcl}
        \bb ^\prime (t)&=&C^{-1}\bf(t)+C^{-1}\Big(
          \left[A_{0b},~A_{0g}\right]\left[\begin{array}{cc}
            A_{bb} & A_{bg}\\
            A_{gb} & A_{gg}
          \end{array}\right]^{-1}
          \left[\begin{array}{c}
            A_{b0}\\
            A_{g0}
          \end{array}\right]-A_{00}\Big)\bb(t),\qquad\forall t\in (0,T],\\
        \bb(0)&=&\hat\bb.
      \end{array}
    \right.
  \end{align}

  In addition, we know
  \begin{align*}
    C_{ij}&=(\varphi _j,\varphi _i)\leq \norm{\varphi _j}\norm{\varphi _i},\\
    (A_{00})_{ij}&=(\WL\{\varphi _j,0,0\},\WL\{\varphi _i,0,0\})\leq\trb{\{\varphi _j,0,0\}}\trb{\{\varphi _i,0,0\}},\\
    (A_{0b})_{ij}&=(\WL\{0,\psi _j,0\},\WL\{\varphi _i,0,0\})\leq\trb{\{0,\psi _j,0\}}\trb{\{\varphi _i,0,0\}},\\
    (A_{0g})_{ij}&=(\WL\{0,0,\varsigma _j\},\WL\{\varphi _i,0,0\})\leq\trb{\{0,0,\varsigma _j\}}\trb{\{\varphi _i,0,0\}},\\
    (A_{bb})_{ij}&=(\WL\{0,\psi _j,0\},\WL\{0,\psi _i,0\})\leq\trb{\{0,\psi _j,0\}}\trb{\{0,\psi _i,0\}},\\
    (A_{bg})_{ij}&=(\WL\{0,0,\varsigma _j\},\WL\{0,\psi _i,0\})\leq\trb{\{0,0,\varsigma _j\}}\trb{\{0,\psi _i,0\}},\\
    (A_{bg})_{ij}&=(\WL\{0,0,\varsigma _j\},\WL\{0,0,\varsigma _i\})\leq\trb{\{0,0,\varsigma _j\}}\trb{\{0,0,\varsigma _i\}},\\
    \bf _i(t)&=(f(t,x),\varphi _i)\leq\norm{f(t,x)}\norm{\varphi _i},
  \end{align*}
  which shows the matrices in (\ref{problem-mat2}) are well defined. The existence and uniqueness of the problem is obtained from the traditional ODE theory, then we arrive at the existence and uniqueness of the solution of the semi-discrete formulation (\ref{semi-scheme}).
\end{proof}

\begin{lemma}\label{lemma-Gronwall}
  (Gronwall lemma\cite{quarteroni_numerical_1994})Suppose $w(t)\in L^1([t_0,\overline{t}])$ is a non-negative function, $\alpha (t)$, $\beta (t)$ is continuous functions on $[t_0,\overline{t}]$. If $\beta (t)$ satisfies
  \begin{align*}
    \beta (t)\leq \alpha (t)+\int _{t_0}^tw(s)\beta (s)ds, \qquad \forall t\in [t_0,\overline{t}],
  \end{align*}
  and $\alpha (t)$ is non-decreasing, Then
  \begin{align*}
    \beta (t)\leq\alpha (t)\exp\Big(\int _{t_0}^tw (s)ds\Big)
  \end{align*}
\end{lemma}
The following theorem shows the stability of the semi-discrete scheme (\ref{semi-scheme}).
\begin{theorem}
  There exists a positive constant $C$, such that
  \begin{align}\label{semi-stab}
    \norm{u_h(t)}^2\leq C\Big(\norm{u_h(0)}^2+\int _0^t\norm{f(s)}^2ds\Big).
  \end{align}
\end{theorem}

\begin{proof}
  Selecting $v=u_h(t)$ in (\ref{semi-scheme}), we have
  \begin{align*}
    (u_{h,t},u_h)+(\WL u_h,\WL u_h)=(f,u_h).
  \end{align*}
  Since $(u_{h,t},u_h)=\frac{1}{2}\frac{d}{dt}\norm{u_h}^2$ and $(\WL u_h,\WL u_h)=\trb{u_h}^2\geq 0$, we get
  \begin{align*}
    \frac{1}{2}\frac{d}{dt}\norm{u_h}^2\leq &(f,u_h)\\
    \leq &\frac{1}{2}\Big(\int _\Omega f^2dx +\int _\Omega u_h^2dx\Big).
  \end{align*}

  Integrating with $t$ on the both sides of the inequality, we find
  \begin{align*}
    \norm{u_h(t)}^2\leq\norm{u_h(0)}^2+\int _0^t\norm{f(s)}^2ds+\int _0^t\norm{u_h(s)}^2ds.
  \end{align*}

  Now we use the Lemma \ref{lemma-Gronwall}. Let $w(t)$ be $1$, $\alpha (t)$ be $\norm{u_h(0)}^2+\int _0^t\norm{f(s)}^2ds$, $\beta (t)=\norm{u_h(t)}^2$ and $t_0=0$,
  \begin{align*}
    \norm{u_h(t)}^2 = &\beta (t)\\
    \leq &\alpha (t)\exp\Big( {\int _0^tw (s)ds}\Big)\\
    \leq &\Big(\norm{u_h(0)}^2+\int _0^t\norm{f(s)}^2ds\Big)\exp\Big( {\int _0^t1ds}\Big)\\
    \leq &\exp{(t)}\Big(\norm{u_h(0)}^2+\int _0^t\norm{f(s)}^2ds\Big)\\
    \leq &C\Big(\norm{u_h(0)}^2+\int _0^t\norm{f(s)}^2ds\Big).
  \end{align*}
\end{proof}

\begin{theorem}
  The numerical solution of the semi-discrete scheme (\ref{semi-scheme}) is unique.
\end{theorem}
\begin{proof}
  Suppose $f=0$ and $\psi =0$, and we only verify $u_h(t)=0$. From (\ref{semi-stab}), we can get $\norm{u_h(t)}^2\leq 0$. Therefore, $u_h(t)=0$ holds true.
\end{proof}

Since the existence and the uniqueness of the full-discrete scheme are equivalent, we only need to prove the uniqueness.
\begin{theorem}
  The full-discrete numerical scheme (\ref{full-scheme}) has an unique solution.
\end{theorem}
\begin{proof}
  Suppose $U_1^n$, $U_2^n$ are the numerical solutions of the full-discrete scheme (\ref{full-scheme}), and let $E^n=U_1^n-U_2^n$, then we have $E^n\in V_h^0$, $E^0=0$ and
  \begin{align*}
    (\overline{\partial}E^n,v_0)+(\WL (\theta E^n+(1-\theta)E^{n-1}),\WL v)=0,\qquad\forall v\in V_h^0.
  \end{align*}

  Select $v=\theta E^n+(1-\theta)E^{n-1}$ in the above identity, we have
  \begin{align*}
    \theta (\overline{\partial} E^n,E^n)+(1-\theta)(\overline{\partial}E^n,E^{n-1})+\trb{\theta E^n+(1-\theta)E^{n-1}}^2=0.
  \end{align*}
  And from
  \begin{align*}
    (\overline{\partial}E^n,E^n)=&\ \frac{1}{\tau}(E^n-E^{n-1},E^n)\\
    =&\ \frac{1}{2\tau}(\norm{E^n}^2-\norm{E^{n-1}}^2+\norm{E^n-E^{n-1}}^2),
  \end{align*}
  and
  \begin{align*}
    (\overline{\partial}E^n,E^{n-1})=&\ \frac{1}{\tau}(E^n-E^{n-1},E^{n-1})\\
    =&\ -\frac{1}{2\tau}(\norm{E^{n-1}}^2-\norm{E^n}^2+\norm{E^{n-1}-E^n}^2),
  \end{align*}
  we arrive at 
  \begin{align*}
    \frac{1}{2\tau}(\norm{E^n}^2-\norm{E^{n-1}}^2+(2\theta -1)\norm{E^n-E^{n-1}}^2)+\trb{\theta E^n+(1-\theta)E^{n-1}}^2=0.
  \end{align*}
  
  Since $(2\theta -1)\norm{E^n-E^{n-1}}^2\geq 0$, we have
  \begin{align*}
    \frac{1}{2\tau}(\norm{E^n}^2-\norm{E^{n-1}}^2)+\trb{\theta E^n+(1-\theta)E^{n-1}}^2\leq 0,
  \end{align*}
  where we use the fact $\theta\in[\frac{1}{2},1]$. And further, we get
  \begin{align*}
    \frac{1}{2\tau}(\norm{E^n}^2-\norm{E^{0}}^2)+\sum _{j=1}^n\trb{\theta E^j+(1-\theta)E^{j-1}}^2\leq 0.
  \end{align*}

  By $\norm{E^0}=0$, we obtain
  \begin{align*}
    \norm{E^n}^2+2\tau\sum _{j=1}^n\trb{\theta E^j+(1-\theta)E^{j-1}}^2\leq 0.
  \end{align*}

  Using the positive property of the norms $\norm{\cdot}$ and $\trb\cdot$, we have the following results
  \begin{align}
    E^n=0,
  \end{align}
  \begin{align}
    \theta E^j+(1-\theta)E^{j-1}=0,\qquad (1\leq j\leq n).
  \end{align}

  When $\theta \neq 1$, from $E^n=0$ and $\theta E^n+(1-\theta)E^{n-1}=0$, we can get $E^{n-1}=0$. Similiarly, with $E^{n-1}=0$ and $\theta E^{n-1}+(1-\theta)E^{n-2}=0$, $E^{n-2}=0$ can be derived. $\cdots\cdots$ Finally, we find
  \begin{align*}
    E^j=0,\qquad (1\leq j\leq n).
  \end{align*}

  When $\theta =1$, we can get the above result directly.
  In other words,
  \begin{align*}
    U_1^j=U_2^j,\qquad (1\leq j\leq n),
  \end{align*}
  which completes the proof.
\end{proof}

Before confirming the stability of the full-discrete scheme, we first derive the following important conclusion.

\begin{lemma}\label{lemma-L2-trb}
  \begin{align}\label{L2-trb}
    \norm{v}\leq C\trb{v},\qquad\forall v\in V_h^0,
  \end{align}
  where $C$ is a positive constant independent of $h$.
\end{lemma}
\begin{proof}
  From the proof of the Lemma \ref{lemma-norm-equ}, we know
  \begin{align}\label{laplace-trb}
    \sumT{\norm{\Delta v_0}_T^2}\leq C\trb{v}^2,\qquad \forall v\in V_h.
  \end{align}
  And the inequalities in Lemma A.6 and Lemma A.7 in \cite{BiharmonicWGReOrder} also holds true, which lead to the estimates in Lemma A.8 in \cite{BiharmonicWGReOrder} hold true. That is to say
  \begin{align}\label{grad-trb}
    (\sumT\norm{\G v_0}_T^2)^\frac{1}{2}\leq C\trb v,\qquad v\in V_h^0,
  \end{align}
  holds.
  From Lemma A.6 in \cite{BiharmonicWGReOrder}, we have
  \begin{align}
    \norm{v_0}^2\leq&\  C(\sumT\norm{\G v_0}_T^2+h^{-1}\sumT{\norm{v_0-v_b}_{\partial T}^2}),\qquad\forall v\in V_h^0\\
    \leq&\  C\trb{v}^2+Ch^{-3}\sumT{\norm{v_0-v_b}_{\partial T}^2}\\
    \leq&\  C\trb{v}^2+C\norm{v}_{2,h}^2\\
    \leq&\  C\trb{v}^2,
  \end{align}
  where we utilize (\ref{grad-trb}) and the Lemma \ref{lemma-norm-equ}.
\end{proof}

\begin{theorem}
  Let $U^n$ be the numerical solution of the full-discrete formulation (\ref{full-scheme}) and $\norm{f(t)}$ be bounded on $[0,\overline{t}]$, then there exists a constant $C>0$ such that
  \begin{align}\label{full-stab}
    \norm{U^n}\leq\norm{U^0}+C\sup _{t\in [0,\overline{t}]}\norm{f(t)}.
  \end{align}
\end{theorem}

\begin{proof}
  Choosing $v=\theta U^n+(1-\theta)U^{n-1}$ in (\ref{full-scheme}), we have
  \begin{align*}
    &(\overline{\partial}U^n,\theta U^n+(1-\theta)U^{n-1})+\trb{\theta U^n+(1-\theta)U^{n-1}}^2\\
    =&(\theta f(t_n)+(1-\theta)f(t_{n-1}),\theta U^n+(1-\theta)U^{n-1}).
  \end{align*}
  By $\theta U^n+(1-\theta)U^{n-1}=(\theta -\frac{1}{2})(U^n-U^{n-1})+\frac{1}{2}(U^n+U^{n-1})$, we can get
  \begin{align*}
    &(U^n-U^{n-1},\theta U^n+(1-\theta)U^{n-1})\\
    =&(U^n-U^{n-1},(\theta -\frac{1}{2})(U^n-U^{n-1})+\frac{1}{2}(U^n+U^{n-1}))\\
    =&(\theta -\frac{1}{2})\norm{U^n-U^{n-1}}^2+\frac{1}{2}(U^n-U^{n-1},U^n+U^{n-1})\\
    =&(\theta -\frac{1}{2})\norm{U^n-U^{n-1}}^2+\frac{1}{2}\norm{U^n}^2-\frac{1}{2}\norm{U^{n-1}}^2.
  \end{align*}
  Therefore, we obtain
  \begin{align*}
    &\frac{1}{2}\norm{U^n}^2-\frac{1}{2}\norm{U^{n-1}}^2+(\theta -\frac{1}{2})\norm{U^n-U^{n-1}}^2+\tau\trb{\theta U^n+(1-\theta)U^{n-1}}^2\\
    =&\tau (\theta f(t_n)+(1-\theta)f(t_{n-1}),\theta U^n+(1-\theta)U^{n-1})\\
    \leq&\tau\norm{\theta f(t_n)+(1-\theta)f(t_{n-1})}\norm{\theta U^n+(1-\theta)U^{n-1}}\\
    \leq&\frac{\tau}{4\epsilon}\norm{\theta f(t_n)+(1-\theta)f(t_{n-1})}^2+\epsilon\tau\norm{\theta U^n+(1-\theta)U^{n-1}}^2,
  \end{align*}
  where we use the Cauchy-Schwarz inequality and the Young's inequality.

  Further, let $\epsilon$ be $\frac{1}{4}$ and we have
  \begin{align*}
    &\ \frac{1}{2}\norm{U^n}^2-\frac{1}{2}\norm{U^{n-1}}^2+(\theta -\frac{1}{2})\norm{U^n-U^{n-1}}^2+\tau\trb{\theta U^n+(1-\theta)U^{n-1}}^2\\
    \leq &\ \tau\norm{\theta f(t_n)+(1-\theta)f(t_{n-1})}^2+\frac{\tau}{4}\norm{\theta U^n+(1-\theta)U^{n-1}}^2\\
    \leq &\ \tau\norm{\theta f(t_n)+(1-\theta)f(t_{n-1})}^2+\frac{\tau}{4}\trb{\theta U^n+(1-\theta)U^{n-1}}^2,
  \end{align*}
  where we use the (\ref{L2-trb}).

  Then we get
  \begin{align*}
    \frac{1}{2}\norm{U^n}^2\leq \frac{1}{2}\norm{U^{n-1}}^2+\tau\norm{\theta f(t_n)+(1-\theta)f(t_{n-1})}^2.
  \end{align*}
  It follows that
  \begin{align*}
    \norm{U^n}^2\leq&\ \norm{U^0}^2+C\tau\sum _{j=1}^n\norm{\theta f(t_j)+(1-\theta)f(t_{j-1})}^2\\
    \leq&\ \norm{U^0}^2+C\tau\sum _{j=1}^n(\theta \norm{f(t_j)}+(1-\theta)\norm{f(t_{j-1})})^2\\
    \leq&\ \norm{U^0}^2+C\tau\sum _{j=1}^n\Big(\theta \sup _{t\in [0,\overline{t}]}\norm{f(t)}+(1-\theta)\sup _{t\in [0,\overline{t}]}\norm{f(t)}\Big)^2\\
    \leq &\ \norm{U^0}^2+C\overline{t}\sup _{t\in [0,\overline{t}]}\norm{f(t)}^2\\
    \leq &\ \norm{U^0}^2+C\sup _{t\in [0,\overline{t}]}\norm{f(t)}^2,
  \end{align*}
  which implies (\ref{full-stab}).
\end{proof}

\section{Error estimations}
In this section, we build estimations about the error $e_h=u_h-Q_hu$. For simplicity, let $\varepsilon _h =u-Q_hu$.

\begin{lemma}\label{operator-ex}\cite{BiharmonicSFWG}
  For any $\varphi\in H^2(\Omega)$, we have 
  \begin{align}
    \WL \varphi =\dQ_h\Delta\varphi, \qquad\forall T\in \T_h.
  \end{align}
\end{lemma}

\begin{theorem}\label{theorem-semi-err-equ}
  We have the following semi-discrete error equation 
  \begin{align}\label{semi-err-equ}
    (e_{h,t},v_0)_{\T_h}+(\WL e_h,\WL v)_{\T_h}=(\WL \varepsilon _h,\WL v)_{\T_h}+l_1(u,v)-l_2(u,v),\qquad\forall v\in V_h^0,
  \end{align}
  where 
  \begin{align}
    \label{l1-def}l_1(u,v)=&~\langle(\G (\Delta u)-\G (\dQ_h\Delta u))\cdot\bn,v_0-v_b\rangle_{\partial\T_h},\\
    \label{l2-def}l_2(u,v)=&~\langle\Delta u-\dQ_h\Delta u,(\G v_0-v_n\bn _e)\cdot\bn\rangle_{\partial\T_h}.
  \end{align}
\end{theorem}
\begin{proof}
  Multiply both sides of the equation $u_t+\Delta ^2u=f$ by $v\in V_h^0$ and integrate over $\Omega$ to obtain
\begin{align*}
  (f,v)_{\T_h}=&(u_t,v)_{\T_h}+(\Delta ^2u,v)_{\T_h}\\
=&(u_t,v_0)_{\T_h}+(\Delta ^2u,v_0)_{\T_h}\\
=&(Q_0u_t,v_0)_{\T_h}+(\Delta u,\Delta v_0)_{\T_h}-\langle\Delta u,\G v_0\cdot\bn\rangle_{\partial\T_h}+\langle\G (\Delta u)\cdot\bn,v_0\rangle_{\partial\T_h}\\
=&(Q_0u_t,v_0)_{\T_h}+(\dQ_h\Delta u,\Delta v_0)_{\T_h}-\langle\Delta u,(\G v_0-v_n\bn _e)\cdot\bn\rangle_{\partial\T_h}+\langle\G (\Delta u)\cdot\bn,v_0-v_b\rangle_{\partial\T_h}\\
=&(Q_0u_t,v_0)_{\T_h}+(v_0,\Delta (\dQ_h\Delta u))_{\partial\T _h}-\langle v_0,\G (\dQ_h\Delta u)\cdot\bn\rangle _{\partial \T_h}+\langle \G v_0\cdot\bn,\dQ_h\Delta u\rangle _{\partial\T_h}\\
&-\langle\Delta u,(\G v_0-v_n\bn _e)\cdot\bn\rangle_{\partial\T_h}+\langle\G (\Delta u)\cdot\bn,v_0-v_b\rangle_{\partial\T_h}\\
=&(Q_0u_t,v_0)_{\T_h}+(\Delta _w v,\dQ_h\Delta u)_{\T_h}+\langle v_b,\G (\dQ_h\Delta u)\cdot\bn\rangle_{\partial\T_h}-\langle v_n\bn _e\cdot\bn,\dQ_h\Delta u\rangle_{\partial\T_h}\\
&-\langle v_0,\G (\dQ_h\Delta u)\cdot\bn\rangle _{\partial \T_h}+\langle \G v_0\cdot\bn,\dQ_h\Delta u\rangle _{\partial\T_h}-\langle\Delta u,(\G v_0-v_n\bn _e)\cdot\bn\rangle_{\partial\T_h}\\
&+\langle\G (\Delta u)\cdot\bn,v_0-v_b\rangle_{\partial\T_h}\\
=&(Q_0u_t,v_0)_{\T_h}+(\Delta _w v,\dQ_h\Delta u)_{\T_h}-\langle v_0-v_b,\G (\dQ_h\Delta u)\cdot\bn\rangle _{\partial \T_h}+\langle (\G v_0-v_n\bn _e)\cdot\bn,\dQ_h\Delta u\rangle _{\partial\T_h}\\
&-\langle\Delta u,(\G v_0-v_n\bn _e)\cdot\bn\rangle_{\partial\T_h}+\langle\G (\Delta u)\cdot\bn,v_0-v_b\rangle_{\partial\T_h}\\
=&(Q_0u_t,v_0)_{\T_h}+(\Delta _w v,\Delta _w u)_{\T_h}+l_1(u,v)-l_2(u,v),
\end{align*}
where we use the definitions of $Q_0$ and $\dQ_h$, the integration by parts, the definition of $\Delta _w$ and the Lemma \ref{operator-ex}.

By the (\ref{semi-scheme}), we have 
\begin{align*}
  (e_{h,t},v_0)_{\T_h}+(\WL (u_h-u),\WL v)_{\T_h}=l_1(u,v)-l_2(u,v).
\end{align*}
Further, we get (\ref{semi-err-equ}).
\end{proof}

\begin{lemma}\cite{BiharmonicSFWG}
  Let $k\geq 2$ and $w\in H^{\m}(\Omega)$, then there exists a constant C such that
  \begin{align}
    \label{laplace-u}\Big(\sumT h_T\norm{\Delta w-\dQ_h\Delta w}^2_{\partial T}\Big)^{\frac{1}{2}}\leq&~ Ch^{k-1+\delta _{k,2}}\norm{w}_{\m}\\
    \label{grad-laplace-u}\Big(\sumT h_T^3\norm{\G (\Delta w-\dQ_h\Delta w)}_{\partial T}^2\Big)^{\frac{1}{2}}\leq&~ Ch^{k-1+\delta _{k,2}}\norm{w}_{\m},
  \end{align}
  where 
  \begin{align*}
    \delta _{i,j}=
    \left\{
\begin{array}{rcl}
1,\qquad &i=j,\\
0,\qquad &i\neq j.
\end{array}
\right.
  \end{align*}
\end{lemma}

\begin{lemma}\cite{BiharmonicSFWG}
  Let $k\geq 2$ and $w\in H^{\max{\{k+1,4\}}}(\Omega)$, we have the following estimations
  \begin{align}
    \label{trb-u-Qhu}&\trb{w-Q_hw}\leq Ch^{k-1}\norm{w}_{k+1},\\
    \label{l1-est}&|l_1(w,v)|\leq Ch^{k-1+\delta _{k,2}}\norm{w}_{\m}\trb{v},\\
    \label{l2-est}&|l_2(w,v)|\leq Ch^{k-1+\delta _{k,2}}\norm{w}_{\m}\trb v,
  \end{align}
  for any $v\in V_h$.
\end{lemma}

\begin{theorem}
  For $u\in H^{\max{\{k+1,4\}}}(\Omega)$, $k\geq 2$, we can obtain
  \begin{align}\label{semi-L2}
    \norm{e_h}^2+\int _0^t\trb{e_h}^2ds\leq\norm{e_h(0,\cdot)}^2+Ch^{2(k-1)}\int_0^t\norm{u}_{\m}^2ds
  \end{align}
  \begin{align}\label{semi-trb}
    &~\norm{e_h}^2+\int _0^t\norm{e_{h,s}}^2ds+\frac{1}{4}\trb{e_h}^2\\
    \leq&~\norm{e_h(0,\cdot)}^2+\trb{e_h(0,\cdot)}^2+Ch^{2(k-1)}\Big[\norm{u(0,\cdot)}_{\m}^2+\norm{u}_{\m}^2\\
    &~+\int _0^t\norm{u}_{\m}^2ds+\int _0^t\norm{u_s}_{\m}^2ds\Big].
  \end{align}
\end{theorem}
\begin{proof}
  Substitute $v=e_h\in V_h^0$ into the error equation (\ref{semi-err-equ}), and we have 
  \begin{align*}
    (e_{h,t},e_h)+(\WL e_h,\WL e_h)=(\WL \varepsilon _h,\WL e_h)+l_1(u,e_h)-l_2(u,e_h),
  \end{align*}
  which implies
  \begin{align*}
    \frac{1}{2}\frac{d}{dt}\norm{e_h}^2+\trb{e_h}^2\leq&~ \trb{\varepsilon _h}\trb{e_h}+Ch^{k-1+\delta _{k,2}}\norm{u}_{\m}\trb{e_h}\\
    \leq &~Ch^{k-1}\norm{u}_{k+1}\trb{e_h}+Ch^{k-1+\delta _{k,2}}\norm{u}_{\m}\\
    \leq &~Ch^{k-1}\norm{u}_{\m}\trb{e_h}\\
    \leq&~ Ch^{2(k-1)}\norm{u}_{\m}^2+\frac{1}{2}\trb{e_h}^2.
  \end{align*}

  That is 
  \begin{align*}
    \frac{d}{dt}\norm{e_h}^2+\trb{e_h}^2\leq&~Ch^{2(k-1)}\norm{u}_{\m}^2.
  \end{align*}

  With regard to $t$, the integral of both sides of this inequality is
  \begin{align*}
    \norm{e_h}^2+\int _0^t\trb{e_h}^2ds\leq\norm{e_h(0,\cdot)}^2+Ch^{2(k-1)}\int_0^t\norm{u}_{\m}^2ds
  \end{align*}

  Replace $v=e_{h,t}$ in the error equation (\ref{semi-err-equ}), we get
  \begin{align*}
    \norm{e_{h,t}}^2+(\WL e_h,\WL e_{h,t})=(\WL \varepsilon _h,\WL e_{h,t})+l_1(u,e_{h,t})-l_2(u,e_{h,t}).
  \end{align*}

  That is
  \begin{align*}
    &~\norm{e_{h,t}}^2+\frac{1}{2}\frac{d}{dt}a(e_h,e_h)\\
    =&~(\WL \varepsilon _h,\WL e_{h,t})+\langle(\G (\Delta u)-\G (\dQ_h\Delta u))\cdot\bn,e_{0,t}-e_{b,t}\rangle_{\partial\T_h}-\langle\Delta u-\dQ_h\Delta u,(\G e_{0,t}-e_{n,t}\bn _e)\cdot\bn\rangle_{\partial\T_h}\\
    =&~\frac{d}{dt}(\WL\varepsilon _h,\WL e_h)_{\T _h}+\frac{d}{dt}\langle(\G (\Delta u)-\G (\dQ_h\Delta u))\cdot\bn,e_{0}-e_{b}\rangle_{\partial\T_h}-\frac{d}{dt}\langle\Delta u-\dQ_h\Delta u,(\G e_{0}-e_{n}\bn _e)\cdot\bn\rangle_{\partial\T_h}\\
    &-(\WL\varepsilon _{h,t},\WL e_h)_{\T _h}-\langle(\G (\Delta u_t)-\G (\dQ_h\Delta u_t))\cdot\bn,e_{0}-e_{b}\rangle_{\partial\T_h}+\langle\Delta u_t-\dQ_h\Delta u_t,(\G e_{0}-e_{n}\bn _e)\cdot\bn\rangle_{\partial\T_h}\\
    \leq&~\frac{d}{dt}(\WL\varepsilon _h,\WL e_h)_{\T _h}+\frac{d}{dt}\langle(\G (\Delta u)-\G (\dQ_h\Delta u))\cdot\bn,e_{0}-e_{b}\rangle_{\partial\T_h}-\frac{d}{dt}\langle\Delta u-\dQ_h\Delta u,(\G e_{0}-e_{n}\bn _e)\cdot\bn\rangle_{\partial\T_h}\\
    &+\frac{3}{4}\trb{\varepsilon _{h,t}}^2+\frac{1}{3}\trb{e_h}^2+\frac{3}{4}\sumT h_T^3\norm{\G (\Delta u_t)-\G (\dQ_h\Delta u_t)}_{\partial T}^2+\frac{1}{3}\sumT h_T^{-3}\norm{e_0-e_b}_{\partial T}^2\\
    &+\frac{3}{4}\sumT h_T\norm{\Delta u_t-\dQ_h\Delta u_t}_{\partial T}^2+\frac{1}{3}\sumT h_T^{-1}\norm{(\G e_0-e_n\bn _e)\cdot\bn}_{\partial T}^2\\
    \leq&~\frac{d}{dt}(\WL\varepsilon _h,\WL e_h)_{\T _h}+\frac{d}{dt}\langle(\G (\Delta u)-\G (\dQ_h\Delta u))\cdot\bn,e_{0}-e_{b}\rangle_{\partial\T_h}-\frac{d}{dt}\langle\Delta u-\dQ_h\Delta u,(\G e_{0}-e_{n}\bn _e)\cdot\bn\rangle_{\partial\T_h}\\
    &+\frac{3}{4}\trb{\varepsilon _{h,t}}^2+\frac{1}{3}\trb{e_h}^2+\frac{3}{4}\sumT h_T^3\norm{\G (\Delta u_t)-\G (\dQ_h\Delta u_t)}_{\partial T}^2+\frac{1}{3}\trb{e_h}^2\\
    &+\frac{3}{4}\sumT h_T\norm{\Delta u_t-\dQ_h\Delta u_t}_{\partial T}^2+\frac{1}{3}\trb{e_h}^2\\
    \leq&~\frac{d}{dt}(\WL\varepsilon _h,\WL e_h)_{\T _h}+\frac{d}{dt}\langle(\G (\Delta u)-\G (\dQ_h\Delta u))\cdot\bn,e_{0}-e_{b}\rangle_{\partial\T_h}-\frac{d}{dt}\langle\Delta u-\dQ_h\Delta u,(\G e_{0}-e_{n}\bn _e)\cdot\bn\rangle_{\partial\T_h}\\
    &+Ch^{2(k-1)}\norm{u_t}_{k+1}^2+Ch^{2(k-1+\delta _{k,2})}\norm{u_t}_{\m}^2+\trb{e_h}^2,
  \end{align*}
  where we use the Cauchy-Schwarz inequality, the Young's inequality, the definition of $\norm{\cdot}_{2,h}$, the Lemma \ref{lemma-norm-equ} and (\ref{grad-laplace-u}-\ref{laplace-u}).

  \begin{align*}
    &~\int _0^t\norm{e_{h,s}}^2ds+\frac{1}{2}a(e_h,e_h)\\
    \leq&~\frac{1}{2}a(e_h(0,\cdot),e_h(0,\cdot))+(\WL\varepsilon _h,\WL e_h)_{\T _h}-(\WL\varepsilon _h(0,\cdot),\WL e_h(0,\cdot))_{\T_h}+l_1(u,e_h)-l_1(u(0,\cdot),e_h(0,\cdot))\\
    &~-l_2(u,e_h)+l_2(u(0,\cdot),e_h(0,\cdot))+Ch^{2(k-1)}\int _0^t\norm{u_s}_{\m}^2ds+\int _0^t\trb{e_h}^2ds\\
    \leq&~\frac{1}{2}a(e_h(0,\cdot),e_h(0,\cdot))+\trb{\varepsilon _h}\trb{e_h}+\trb{\varepsilon _h(0,\cdot)}\trb{e_h(0,\cdot)}+|l_1(u,e_h)|+|l_1(u(0,\cdot),e_h(0,\cdot))|\\
    &~+|l_2(u,e_h)|+|l_2(u(0,\cdot),e_h(0,\cdot))|+Ch^{2(k-1)}\int _0^t\norm{u_s}_{\m}^2ds+\int _0^t\trb{e_h}^2ds\\
    \leq&~\frac{1}{2}a(e_h(0,\cdot),e_h(0,\cdot))+Ch^{2(k-1)}\norm{u}_{k+1}^2+\frac{1}{12}\trb{e_h}^2+Ch^{2(k-1)}\norm{u(0,\cdot)}_{k+1}^2+\frac{1}{6}\trb{e_h(0,\cdot)}^2\\
    &~+Ch^{2(k-1+\delta _{k,2})}\norm{u}_{\m}^2+\frac{1}{12}\trb{e_h}^2+Ch^{2(k-1+\delta _{k,2})}\norm{u(0,\cdot)}_{\m}^2+\frac{1}{6}\trb{e_h(0,\cdot)}^2\\
    &~+Ch^{2(k-1+\delta _{k,2})}\norm{u}_{\m}^2+\frac{1}{12}\trb{e_h}^2+Ch^{2(k-1+\delta _{k,2})}\norm{u(0,\cdot)}_{\m}^2+\frac{1}{6}\trb{e_h(0,\cdot)}^2\\
    &~+Ch^{2(k-1)}\int _0^t\norm{u_s}_{\m}^2ds+\int _0^t\trb{e_h}^2ds\\
    \leq&~\trb{e_h(0,\cdot)}^2+\frac{1}{4}\trb{e_h}^2+Ch^{2(k-1)}\norm{u(0,\cdot)}_{\m}^2+Ch^{2(k-1)}\norm{u}_{\m}^2\\
    &~+Ch^{2(k-1)}\int _0^t\norm{u_s}_{\m}^2ds+\int _0^t\trb{e_h}^2ds,
  \end{align*}

  Therefore, we have
  \begin{align*}
    &~\int _0^t\norm{e_{h,s}}^2ds+\frac{1}{4}\trb{e_h}^2\\
    \leq&~\trb{e_h(0,\cdot)}^2+Ch^{2(k-1)}\Big(\norm{u(0,\cdot)}_{\m}^2+\norm{u}_{\m}^2+\int _0^t\norm{u_s}_{\m}^2ds\Big)+\int _0^t\trb{e_h}^2ds.
  \end{align*}

  Combing with (\ref{semi-L2}), we get
  \begin{align*}
    ~&\norm{e_h}^2+\int _0^t\norm{e_{h,s}}^2ds+\frac{1}{4}\trb{e_h}^2\\
    \leq&~\norm{e_h(0,\cdot)}^2+\trb{e_h(0,\cdot)}^2+Ch^{2(k-1)}\Big[\norm{u(0,\cdot)}_{\m}^2+\norm{u}_{\m}^2\\
    ~&+\int _0^t\norm{u}_{\m}^2ds+\int _0^t\norm{u_s}_{\m}^2ds\Big].
  \end{align*}
\end{proof}

\begin{theorem}\label{theorem-full-err-equ}
  $e^n=U^n-Q_hu(t_n)$, we have the full-discrete error equation
  \begin{align}\label{full-err-equ}
    &~(\overline{\partial}e^n,v_0)+(\WL (\theta e^n+(1-\theta)e^{n-1}),\WL v)\\ \nonumber
    =&~(\theta u_t(t_n)+(1-\theta)u_t(t_{n-1})-\overline{\partial}u(t_n),v_0)+l_1(\theta u(t_n)+(1-\theta)u(t_{n-1}),v)\\ \nonumber
    &~-l_2(\theta u(t_n)+(1-\theta)u(t_{n-1}),v)
    +(\Delta _w (\theta (u-Q_hu)(t_n)+(1-\theta)(u-Q_hu)(t_{n-1})),\Delta _w v)
  \end{align}
\end{theorem}
\begin{proof}
  From the proof of the Theorem \ref{theorem-semi-err-equ}, we know 
  \begin{align*}
    (f,v)_{\T_h}=&~(Q_0u_t,v_0)_{\T_h}+(\Delta _w v,\Delta _w u)_{\T_h}+l_1(u,v)-l_2(u,v)\\
    =&~(u_t,v_0)_{\T_h}+(\Delta _w Q_hu,\Delta _w v)_{\T_h}+l_1(u,v)-l_2(u,v)+(\Delta _w (u-Q_hu),\Delta _w v)_{\T_h}
  \end{align*}
  which implies
  \begin{align*}
    (\Delta _w (\theta Q_hu(t_n)+(1-\theta) Q_hu(t_{n-1})),\Delta _w v)=&~(\theta f(t_n)+(1-\theta)f(t_{n-1}),v_0)-(\theta u_t(t_n)+(1-\theta)u_t(t_{n-1}),v_0)\\
    &~-l_1(\theta u(t_n)+(1-\theta)u(t_{n-1}),v)+l_2(\theta u(t_n)+(1-\theta)u(t_{n-1}),v)\\
    &~-(\Delta _w (\theta (u-Q_hu)(t_n)+(1-\theta)(u-Q_hu)(t_{n-1})),\Delta _w v).
  \end{align*}
  
  Combing with (\ref{full-scheme}), we get
  \begin{align*}
    &~(\overline{\partial}(U^n-Q_hu(t_n)),v_0)+(\WL (\theta (U^n-Q_hu(t_n))+(1-\theta)(U^{n-1}-Q_hu(t_{n-1}))),\WL v)\\
    =&~(\theta f(t_n)+(1-\theta)f(t_{n-1}),v_0)-\Big[(\overline{\partial}u(t_n),v_0)+(\Delta _w (\theta Q_hu(t_n)+(1-\theta) Q_hu(t_{n-1})),\Delta _w v)\Big]\\
    =&~(\theta u_t(t_n)+(1-\theta)u_t(t_{n-1})-\overline{\partial}u(t_n),v_0)+l_1(\theta u(t_n)+(1-\theta)u(t_{n-1}),v)-l_2(\theta u(t_n)+(1-\theta)u(t_{n-1}),v)\\
    &~+(\Delta _w (\theta (u-Q_hu)(t_n)+(1-\theta)(u-Q_hu)(t_{n-1})),\Delta _w v),
  \end{align*}
  which show the (\ref{full-err-equ}).
\end{proof}

\begin{theorem}
  Let $u\in H^{\max{\{k+1,4\}}}(\Omega)$, and we have the error estimations:
  When $\theta\in(\frac{1}{2},1]$,
  \begin{align}\label{err-equ-full-L2}
    &~\norm{e^n}^2+\tau\sum_{j=1}^n\trb{\theta e^j+(1-\theta)e^{j-1}}^2\\ \nonumber
    \leq &~\norm{e^{0}}^2+Ch^{2(k-1)}\norm{u}_{\m,\infty}^2+2\tau ^2\Big(\int _{0}^{t_n}\norm{u_{ss}}^2ds\Big).
  \end{align}
  In particular, when $\theta =1$, we have 
\begin{align}\label{err-equ-full-H2}
  \trb{e^n}^2
  \leq &~C\Big(\norm{e^0}^2+\trb{e^0}^2\Big)+Ch^{2(k-1)}\Big(\norm{u(0)}_{\m}^2+\norm{u}_{\m,\infty}^2+\norm{u_t}_{\m,\infty}^2\\
  &~+\int_{0}^{t_n}\norm{u_{ss}}_{\m}^2ds\Big)+\tau^2 C\int _{0}^{t_n}\norm{u_{ss}}^2ds.
\end{align}

  When $\theta =\frac{1}{2}$, we obtain
  \begin{align}\label{err-equ-full-L2-2}
    &~\norm{e^n}^2+\tau\sum_{j=1}^n\trb{\theta e^j+(1-\theta)e^{j-1}}^2\\ \nonumber
    \leq &~\norm{e^{0}}^2+Ch^{2(k-1)}\norm{u}_{\m,\infty}^2+\frac{\tau ^4}{2}\Big(\int _{0}^{t_n}\norm{u_{sss}}^2ds\Big),
  \end{align}
  and
  \begin{align}\label{err-equ-full-H2-2}
    \trb{e^n}^2
    \leq &~C\Big(\norm{e^0}^2+\trb{e^0}^2\Big)+Ch^{2(k-1)}\Big(\norm{u(0)}_{\m}^2+\norm{u}_{\m,\infty}^2+\norm{u_t}_{\m,\infty}^2\\
    &~+\int_{0}^{t_n}\norm{u_{ss}}_{\m}^2ds\Big)+\tau^4 C\int _{0}^{t_n}\norm{u_{sss}}^2ds,
  \end{align}
  where $\norm{u}_{\m,\infty}=\max _{0\leq j\leq n}\norm{u(t_j)}_{\m}$ and $\norm{u_t}_{\m,\infty}=\max _{0\leq j\leq n}\norm{u_t(t_j)}_{\m}$.
\end{theorem}
\begin{proof}
  Selecting $v=\theta e^n+(1-\theta)e^{n-1}$ in (\ref{full-err-equ}), we have
  \begin{align*}
    &~(\overline{\partial}e^n,\theta e^n+(1-\theta)e^{n-1})+(\WL (\theta e^n+(1-\theta)e^{n-1}),\WL (\theta e^n+(1-\theta)e^{n-1}))\\ 
    =&~(\theta u_t(t_n)+(1-\theta)u_t(t_{n-1})-\overline{\partial}u(t_n),\theta e^n+(1-\theta)e^{n-1})+l_1(\theta u(t_n)+(1-\theta)u(t_{n-1}),\theta e^n+(1-\theta)e^{n-1})\\ 
    &~-l_2(\theta u(t_n)+(1-\theta)u(t_{n-1}),\theta e^n+(1-\theta)e^{n-1})\\
    &~+(\Delta _w (\theta (u-Q_hu)(t_n)+(1-\theta)(u-Q_hu)(t_{n-1})),\Delta _w (\theta e^n+(1-\theta)e^{n-1})).
  \end{align*}

  Since $(\overline{\partial}e^n,\theta e^n+(1-\theta)e^{n-1})=\frac{1}{\tau}(\theta -\frac{1}{2})\norm{e^n-e^{n-1}}^2+\frac{1}{2\tau}\norm{e^n}^2-\frac{1}{2\tau}\norm{e^{n-1}}^2$ and the Cauchy-Schwarz inequality, we know
  \begin{align*}
    &~\frac{1}{\tau}(\theta -\frac{1}{2})\norm{e^n-e^{n-1}}^2+\frac{1}{2\tau}\norm{e^n}^2-\frac{1}{2\tau}\norm{e^{n-1}}^2+ \trb{\theta e^n+(1-\theta)e^{n-1}}^2\\
    \leq &~\norm{\theta u_t(t_n)+(1-\theta)u_t(t_{n-1})-\overline{\partial}u(t_n)}\norm{\theta e^n+(1-\theta)e^{n-1}}\\
    &~+|l_1(\theta u(t_n)+(1-\theta)u(t_{n-1}),\theta e^n+(1-\theta)e^{n-1})|\\ 
    &~+|l_2(\theta u(t_n)+(1-\theta)u(t_{n-1}),\theta e^n+(1-\theta)e^{n-1})|\\
    &~+\trb{\theta (u-Q_hu)(t_n)+(1-\theta)(u-Q_hu)(t_{n-1})}\trb{\theta e^n+(1-\theta)e^{n-1}}.
  \end{align*}

  When $\theta\in (\frac{1}{2},1]$,
  \begin{align}
    \label{part_1}&\norm{\OP{u(t_n)-(\theta u_t(t_n)+(1-\theta)u_t(t_{n-1}))}}^2\\ \nonumber
    =&\int _\Omega\frac{1}{\tau ^2}\Big(\int _{t_{n-1}}^{t_n}(s-(1-\theta)t_n-\theta t_{n-1})u_{ss}ds\Big)^2dx\\ \nonumber
    \leq&\int _\Omega\frac{1}{\tau ^2}\Big(\int _{t_{n-1}}^{t_n}(s-(1-\theta)t_n-\theta t_{n-1})^2ds\Big)\Big(\int _{t_{n-1}}^{t_n}u_{ss}^2ds\Big)dx\\ \nonumber
    \leq&\int _\Omega \theta ^2\tau\int_{t_{n-1}}^{t_n}u_{ss}^2dsdx\\ \nonumber
    \leq& \theta ^2\tau\int _{t_{n-1}}^{t_n}\norm{u_{ss}}^2ds\\ \nonumber
    \leq& \tau\int _{t_{n-1}}^{t_n}\norm{u_{ss}}^2ds, \nonumber
  \end{align}
  and when $\theta = \frac{1}{2}$,
  \begin{align}
    \label{part_2}&\norm{\OP{u(t_n)-\frac{1}{2}(u_t(t_n)+u_t(t_{n-1}))}}^2\\ \nonumber
    =&\int _\Omega\frac{1}{4\tau ^2}\Big(\int _{t_{n-1}}^{t_n}(t_n-s)(t_{n-1}-s)u_{sss}ds\Big)^2dx\\ \nonumber
    \leq&\int _\Omega\frac{1}{4\tau ^2}\Big(\int_{t_{n-1}}^{t_n}(t_n-s)^2(t_{n-1}-s)^2ds\Big)\Big(\int _{t_{n-1}}^{t_n}u_{sss}^2ds\Big)dx\\ \nonumber
    \leq&\frac{1}{4}\tau ^3\int _\Omega\int _{t_{n-1}}^{t_n}u_{sss}^2dsdx\\ \nonumber
    \leq&\frac{1}{4}\tau ^3\int _{t_{n-1}}^{t_n}\norm{u_{sss}}^2ds. \nonumber
  \end{align}

  By (\ref{l1-est}-\ref{l2-est}), we have
  \begin{align*}
    &~|l_1(\theta u(t_n)+(1-\theta)u(t_{n-1}),\theta e^n+(1-\theta)e^{n-1})|\\
    \leq &~Ch^{k-1+\delta _{k,2}}\Big(\theta \norm{u(t_n)}_{\m}+(1-\theta)\norm{u(t_{n-1})}_{\m}\Big)\trb{\theta e^n+(1-\theta)e^{n-1}}\\
    &~|l_2(\theta u(t_n)+(1-\theta)u(t_{n-1}),\theta e^n+(1-\theta)e^{n-1})|\\
    \leq &~Ch^{k-1+\delta _{k,2}}\Big(\theta \norm{u(t_n)}_{\m}+(1-\theta)\norm{u(t_{n-1})}_{\m}\Big)\trb{\theta e^n+(1-\theta)e^{n-1}}
  \end{align*}

  Using (\ref{trb-u-Qhu}), we get
  \begin{align*}
    &~\trb{\theta (u-Q_hu)(t_n)+(1-\theta)(u-Q_hu)(t_{n-1})}\\
    \leq &~\theta\trb{(u-Q_hu)(t_n)}+(1-\theta)\trb{(u-Q_hu)(t_{n-1})}\\
    \leq &~Ch^{k-1}\Big(\theta\norm{u(t_n)}_{k+1}+(1-\theta)\norm{u(t_{n-1})}_{k+1}\Big)
  \end{align*}

  Therefore, when $\theta \in(\frac{1}{2},1]$, there is 
  \begin{align*}
    &~\frac{1}{\tau}(\theta -\frac{1}{2})\norm{e^n-e^{n-1}}^2+\frac{1}{2\tau}\norm{e^n}^2-\frac{1}{2\tau}\norm{e^{n-1}}^2+ \trb{\theta e^n+(1-\theta)e^{n-1}}^2\\
    \leq &~\tau ^{\frac{1}{2}}\Big(\int _{t_{n-1}}^{t_n}\norm{u_{ss}}^2ds\Big)^{\frac{1}{2}}\norm{\theta e^n+(1-\theta)e^{n-1}}+Ch^{k-1}\Big(\theta\norm{u(t_n)}_{\m}+(1-\theta)\norm{u(t_{n-1})}_{\m}\Big)\\
    &~\trb{\theta e^n+(1-\theta)e^{n-1}}\\
    \leq &~\frac{\varepsilon\tau}{4}\Big(\int _{t_{n-1}}^{t_n}\norm{u_{ss}}^2ds\Big)+\frac{1}{\varepsilon}\norm{\theta e^n+(1-\theta)e^{n-1}}^2+\varepsilon Ch^{2(k-1)}\Big(\theta\norm{u(t_n)}_{\m}+(1-\theta)\norm{u(t_{n-1})}_{\m}\Big)^2\\
    &~+\frac{1}{\varepsilon}\trb{\theta e^n+(1-\theta)e^{n-1}}^2\\
    \leq &~\varepsilon Ch^{2(k-1)}\Big(\theta\norm{u(t_n)}_{\m}+(1-\theta)\norm{u(t_{n-1})}_{\m}\Big)^2+\frac{\varepsilon\tau}{4}\Big(\int _{t_{n-1}}^{t_n}\norm{u_{ss}}^2ds\Big)\\
    &~+\frac{2}{\varepsilon}\trb{\theta e^n+(1-\theta)e^{n-1}}^2,
  \end{align*}
  where we use the Young's inequality and (\ref{L2-trb}).

  Setting $\varepsilon = 4$, we have
  \begin{align*}
    &~\frac{1}{\tau}(\theta -\frac{1}{2})\norm{e^n-e^{n-1}}^2+\frac{1}{2\tau}\norm{e^n}^2-\frac{1}{2\tau}\norm{e^{n-1}}^2+\frac{1}{2} \trb{\theta e^n+(1-\theta)e^{n-1}}^2\\
    \leq &~Ch^{2(k-1)}\Big(\theta\norm{u(t_n)}_{\m}+(1-\theta)\norm{u(t_{n-1})}_{\m}\Big)^2+\tau\Big(\int _{t_{n-1}}^{t_n}\norm{u_{ss}}^2ds\Big),
  \end{align*}
  which implies 
  \begin{align*}
    &~\norm{e^n}^2+\tau\trb{\theta e^n+(1-\theta)e^{n-1}}^2\\
    \leq &~\norm{e^{n-1}}^2+\tau Ch^{2(k-1)}\Big(\theta\norm{u(t_n)}_{\m}+(1-\theta)\norm{u(t_{n-1})}_{\m}\Big)^2+2\tau ^2\Big(\int _{t_{n-1}}^{t_n}\norm{u_{ss}}^2ds\Big).
  \end{align*}

  Therefore, we get
  \begin{align*}
    &~\norm{e^n}^2+\tau\sum_{j=1}^n\trb{\theta e^j+(1-\theta)e^{j-1}}^2\\
    \leq &~\norm{e^{0}}^2+\tau Ch^{2(k-1)}\Big(\theta \sum _{j=1}^n\norm{u(t_j)}_{\m}+(1-\theta)\sum_{j=1}^n\norm{u(t_{j-1})}_{\m}\Big)^2+2\tau ^2\Big(\int _{0}^{t_n}\norm{u_{ss}}^2ds\Big)\\
    \leq &~\norm{e^{0}}^2+\tau Ch^{2(k-1)}\Big(\sum _{j=0}^n\norm{u(t_j)}_{\m}\Big)^2+2\tau ^2\Big(\int _{0}^{t_n}\norm{u_{ss}}^2ds\Big)\\
    \leq &~\norm{e^{0}}^2+(n+1)\tau Ch^{2(k-1)}\norm{u}_{\m,\infty}^2+2\tau ^2\Big(\int _{0}^{t_n}\norm{u_{ss}}^2ds\Big)\\
    \leq &~\norm{e^{0}}^2+Ch^{2(k-1)}\norm{u}_{\m,\infty}^2+2\tau ^2\Big(\int _{0}^{t_n}\norm{u_{ss}}^2ds\Big).
  \end{align*}

  When $\theta =\frac{1}{2}$, similarly, we obtain
  \begin{align*}
    &~\norm{e^n}^2+\tau\sum_{j=1}^n\trb{\theta e^j+(1-\theta)e^{j-1}}^2\\
    \leq &~\norm{e^{0}}^2+Ch^{2(k-1)}\norm{u}_{\m,\infty}^2+\frac{\tau ^4}{2}\Big(\int _{0}^{t_n}\norm{u_{sss}}^2ds\Big).
  \end{align*}

  Choosing $v=\overline{\partial}e^n$ in (\ref{full-err-equ}), we know
  \begin{align*}
    &~\norm{\overline\partial e^n}^2+(\WL (\theta e^n+(1-\theta)e^{n-1}),\WL\overline\partial e^n)\\
    =&~(\theta u_t(t_n)+(1-\theta)u_t(t_{n-1})-\overline{\partial}u(t_n),\overline\partial e^n)+l_1(\theta u(t_n)+(1-\theta)u(t_{n-1}),\overline\partial e^n)\\
    &~-l_2(\theta u(t_n)+(1-\theta)u(t_{n-1}),\overline\partial e^n)
    +(\Delta _w (\theta (u-Q_hu)(t_n)+(1-\theta)(u-Q_hu)(t_{n-1})),\Delta _w \overline\partial e^n).
  \end{align*}

  From $\theta e^n+(1-\theta)e^{n-1}=(\theta -\frac{1}{2})(e^n-e^{n-1})+\frac{1}{2}(e^n+e^{n-1})$, we have
  \begin{align}
    \label{left}&~\norm{\overline\partial e^n}^2+(\WL (\theta e^n+(1-\theta)e^{n-1}),\WL\overline\partial e^n)\\ \nonumber
    =&~\frac{1}{\tau ^2}\norm{e^n-e^{n-1}}^2+\frac{1}{\tau}(\theta -\frac{1}{2})\trb{e^n-e^{n-1}}^2+\frac{1}{2\tau}\Big(\trb{e^n}^2-\trb{e^{n-1}}^2\Big).\nonumber
  \end{align}

  For simplicity, we use the following symbols
  \begin{align*}
    w_{1,T}^n=&~[\G (\Delta u(t_n))-\G (\dQ_h\Delta u(t_n))]\cdot\bn|_{\partial T},\\
    e_{1,T}^n=&~(e_0(t_n)-e_b(t_n))|_{\partial T},\\
    w_{2,T}^n=&~[\Delta u(t_n)-\dQ_h\Delta u(t_n)]|_{\partial T},\\
    e_{2,T}^n=&~(\G e_0-e_n\bn_e)\cdot\bn|_{\partial T},\\
    w_{3,T}^n=&~\WL ((u-Q_hu)(t_n))|_T,\\
    e_{3,T}^n=&~\WL e(t_n)|_T.
  \end{align*}

  Therefore, we have
  \begin{align}
    \label{right}&~(\theta u_t(t_n)+(1-\theta)u_t(t_{n-1})-\overline{\partial}u(t_n),\overline\partial e^n)+l_1(\theta u(t_n)+(1-\theta)u(t_{n-1}),\overline\partial e^n)\\ \nonumber
    &~-l_2(\theta u(t_n)+(1-\theta)u(t_{n-1}),\overline\partial e^n)
    +(\Delta _w (\theta (u-Q_hu)(t_n)+(1-\theta)(u-Q_hu)(t_{n-1})),\Delta _w \overline\partial e^n)\\ \nonumber
    =&~\frac{1}{\tau}(\theta u_t(t_n)+(1-\theta)u_t(t_{n-1})-\overline{\partial}u(t_n),e^n-e^{n-1})+\theta\sumT\langle w_{1,T}^n,\overline\partial e_{1,T}^n\rangle_{\partial T}+(1-\theta)\sumT\langle w_{1,T}^{n-1},\overline\partial e_{1,T}^n\rangle_{\partial T}\\ \nonumber
    &~-\theta\sumT\langle w_{2,T}^n,\overline\partial e_{2,T}^n\rangle_{\partial T}-(1-\theta)\sumT\langle w_{2,T}^{n-1},\overline\partial e_{2,T}^n\rangle_{\partial T}\\ \nonumber
    &~+\theta\sumT (w_{3,T}^n,\overline\partial e_{3,T}^n)_{T}+(1-\theta)\sumT (w_{3,T}^{n-1},\overline\partial e_{3,T}^n)_{T}\\ \nonumber
    \leq &~\frac{1}{\tau}\norm{\theta u_t(t_n)+(1-\theta)u_t(t_{n-1})-\overline{\partial}u(t_n)}\norm{e^n-e^{n-1}}\\ \nonumber
    &~+\theta\sumT \Big[\overline\partial\langle w_{1,T}^n,e_{1,T}^n\rangle _{\partial T}+\langle (w_{1,T}^n)_t-\overline\partial w_{1,T}^n,e_{1,T}^{n-1}\rangle _{\partial T}-\langle (w_{1,T}^n)_t,e_{1,T}^{n-1}\rangle _{\partial T}\Big]\\ \nonumber
    &~+(1-\theta)\sumT \Big[\overline\partial\langle w_{1,T}^n,e_{1,T}^n\rangle _{\partial T}+\langle (w_{1,T}^n)_t-\overline\partial w_{1,T}^n,e_{1,T}^{n}\rangle _{\partial T}-\langle (w_{1,T}^n)_t,e_{1,T}^{n}\rangle _{\partial T}\Big]\\ \nonumber
    &~-\theta\sumT \Big[\overline\partial\langle w_{2,T}^n,e_{2,T}^n\rangle _{\partial T}+\langle (w_{2,T}^n)_t-\overline\partial w_{2,T}^n,e_{2,T}^{n-1}\rangle _{\partial T}-\langle (w_{2,T}^n)_t,e_{2,T}^{n-1}\rangle _{\partial T}\Big]\\ \nonumber
    &~-(1-\theta)\sumT \Big[\overline\partial\langle w_{2,T}^n,e_{2,T}^n\rangle _{\partial T}+\langle (w_{2,T}^n)_t-\overline\partial w_{2,T}^n,e_{2,T}^{n}\rangle _{\partial T}-\langle (w_{2,T}^n)_t,e_{2,T}^{n}\rangle _{\partial T}\Big]\\ \nonumber
    &~+\theta\sumT \Big[\overline\partial (w_{3,T}^n,e_{3,T}^n)_{T}+((w_{3,T}^n)_t-\overline\partial w_{3,T}^n,e_{3,T}^{n-1})_{T}-((w_{3,T}^n)_t,e_{3,T}^{n-1})_{T}\Big]\\ \nonumber
    &~+(1-\theta)\sumT \Big[\overline\partial (w_{3,T}^n,e_{3,T}^n) _{T}+((w_{3,T}^n)_t-\overline\partial w_{3,T}^n,e_{3,T}^{n})_{ T}- ((w_{3,T}^n)_t,e_{3,T}^{n})_{T}\Big]\\ \nonumber
    \leq &~\frac{1}{\tau}\Big(\frac{1}{4\varepsilon}\norm{\theta u_t(t_n)+(1-\theta)u_t(t_{n-1})-\overline{\partial}u(t_n)}^2+\varepsilon \norm{e^n-e^{n-1}}^2\Big)\\ \nonumber
    &~+\sumT \Big[\overline\partial\langle w_{1,T}^n,e_{1,T}^n\rangle _{\partial T}+\langle (w_{1,T}^n)_t-\overline\partial w_{1,T}^n,\theta e_{1,T}^{n-1}+(1-\theta)e_{1,T}^{n}\rangle _{\partial T}-\langle (w_{1,T}^n)_t,\theta e_{1,T}^{n-1}+(1-\theta)e_{1,T}^{n}\rangle _{\partial T}\Big]\\ \nonumber
    &~-\sumT \Big[\overline\partial\langle w_{2,T}^n,e_{2,T}^n\rangle _{\partial T}+\langle (w_{2,T}^n)_t-\overline\partial w_{2,T}^n,\theta e_{2,T}^{n-1}+(1-\theta)e_{2,T}^{n}\rangle _{\partial T}-\langle (w_{2,T}^n)_t,\theta e_{2,T}^{n-1}+(1-\theta)e_{2,T}^{n}\rangle _{\partial T}\Big]\\ \nonumber
    &~+ \sumT \Big[\overline\partial (w_{3,T}^n,e_{3,T}^n)_{T}+((w_{3,T}^n)_t-\overline\partial w_{3,T}^n,\theta e_{3,T}^{n-1}+(1-\theta)e_{3,T}^{n})_{T}-((w_{3,T}^n)_t,\theta e_{3,T}^{n-1}+(1-\theta)e_{3,T}^{n})_{T}\Big] \nonumber
  \end{align}
  where we use 
  \begin{align*}
    \langle w_{1,T}^n,\overline\partial e_{1,T}^n\rangle_{\partial T}=&~\overline\partial\langle w_{1,T}^n,e_{1,T}^n\rangle _{\partial T}+\langle (w_{1,T}^n)_t-\overline\partial w_{1,T}^n,e_{1,T}^{n-1}\rangle _{\partial T}-\langle (w_{1,T}^n)_t,e_{1,T}^{n-1}\rangle _{\partial T},\\
    \langle w_{1,T}^{n-1},\overline\partial e_{1,T}^n\rangle_{\partial T}=&~\overline\partial\langle w_{1,T}^n,e_{1,T}^n\rangle _{\partial T}+\langle (w_{1,T}^n)_t-\overline\partial w_{1,T}^n,e_{1,T}^{n}\rangle _{\partial T}-\langle (w_{1,T}^n)_t,e_{1,T}^{n}\rangle _{\partial T}.
  \end{align*}
  and the similar results for $\langle w_{2,T}^n,\overline\partial e_{2,T}^n\rangle_{\partial T}$, $\langle w_{2,T}^{n-1},\overline\partial e_{2,T}^n\rangle_{\partial T}$, $(w_{3,T}^n,\overline\partial e_{3,T}^n)_{T}$, $(w_{3,T}^{n-1},\overline\partial e_{3,T}^n)_{T}$.

  In addition, from the Cauchy-Schwarz inequality, (\ref{grad-laplace-u}), the definition of $\norm{\cdot}_{2,h}$, the Lemma \ref{lemma-norm-equ}, (\ref{part_1}), the Young's inequality and (\ref{grad-laplace-u}), we have 
  \begin{align*}
    &~\sumT \langle (w_{1,T}^n)_t-\overline\partial w_{1,T}^n,\theta e_{1,T}^{n-1}+(1-\theta)e_{1,T}^{n}\rangle _{\partial T}\\
    \leq &~\Big(\sumT h_T^3\norm{(w_{1,T}^n)_t-\overline\partial w_{1,T}^n}_{\partial T}^2\Big)^{\frac{1}{2}}\Big(\sumT h_T^{-3}\norm{\theta e_{1,T}^{n-1}+(1-\theta)e_{1,T}^{n}}_{\partial T}^2\Big)^{\frac{1}{2}}\\
    \leq &~\Big(\sumT h_T^3\norm{(w_{1,T}^n)_t-\overline\partial w_{1,T}^n}_{\partial T}^2\Big)^{\frac{1}{2}}\trb{\theta e^{n-1}+(1-\theta)e^{n}}\\
    \leq &~\Big(C\tau\sumT h_T^3\int _{t_{n-1}}^{t_n}\norm{w_{1,ss}}_{\partial T}^2ds\Big)^{\frac{1}{2}}\trb{\theta e^{n-1}+(1-\theta)e^{n}}\\
    \leq &~\frac{1}{4\varepsilon _{11}}C\tau \sumT h_T^3\int _{t_{n-1}}^{t_n}\norm{w_{1,ss}}_{\partial T}^2ds +\varepsilon _{11}\trb{\theta e^{n-1}+(1-\theta)e^{n}}^2\\
    \leq &~\frac{1}{4\varepsilon _{11}}C\tau h^{2(k-1+\delta _{k,2})}\int_{t_{n-1}}^{t_n}\norm{u_{ss}}_{\m}^2ds+\varepsilon _{11}\trb{\theta e^{n-1}+(1-\theta)e^{n}}^2,
  \end{align*}
  \\
  \begin{align*}
    &~\sumT \langle (w_{1,T}^n)_t,\theta e_{1,T}^{n-1}+(1-\theta)e_{1,T}^{n}\rangle _{\partial T}\\
    \leq &~\Big(\sumT h_T^3\norm{(w_{1,T}^n)_t}_{\partial T}^2\Big)^{\frac{1}{2}}\Big(\sumT h_T^{-3}\norm{\theta e_{1,T}^{n-1}+(1-\theta)e_{1,T}^{n}}_{\partial T}^2\Big)^{\frac{1}{2}}\\
    \leq &~\Big(\sumT h_T^3\norm{(w_{1,T}^n)_t}_{\partial T}^2\Big)^{\frac{1}{2}}\trb{\theta e^{n-1}+(1-\theta)e^{n}}\\
    \leq &~\Big(\sumT h_T^3\norm{[\G (\Delta u_t(t_n))-\G (\dQ_h\Delta u_t(t_n))]\cdot\bn}_{\partial T}^2\Big)^{\frac{1}{2}}\trb{\theta e^{n-1}+(1-\theta)e^{n}}\\
    \leq &~Ch^{k-1+\delta _{k,2}}\norm{u_t(t_n)}_{\m}\trb{\theta e^{n-1}+(1-\theta)e^{n}}\\
    \leq &~\frac{1}{4\varepsilon _{12}}Ch^{2(k-1+\delta _{k,2})}\norm{u_t(t_n)}_{\m}^2+\varepsilon _{12}\trb{\theta e^{n-1}+(1-\theta)e^{n}}^2,
  \end{align*}
    and
  \begin{align*}
    &~\sumT\langle w_{1,T}^n,e_{1,T}^n\rangle _{\partial T}\\
    \leq &~\Big(\sumT h_T^{3} \norm{w_{1,T}^n}_{\partial T}^2\Big)^{\frac{1}{2}}\Big(\sumT h_T^{-3} \norm{e_{1,T}^n}_{\partial T}^2\Big)^{\frac{1}{2}}\\
    \leq &~\Big(\sumT h_T^3\norm{[\G (\Delta u(t_n))-\G (\dQ_h\Delta u(t_n))]\cdot\bn}_{\partial T}^2\Big)^{\frac{1}{2}}\trb{e^{n}}\\
    \leq &~Ch^{k-1+\delta _{k,2}}\norm{u(t_n)}_{\m}\trb{e^n}
  \end{align*}
  In a similar way, we have
  \begin{align*}
    &~\sumT \langle (w_{2,T}^n)_t-\overline\partial w_{2,T}^n,\theta e_{2,T}^{n-1}+(1-\theta)e_{2,T}^{n}\rangle _{\partial T}\\
    \leq &~\frac{1}{4\varepsilon _{21}}C\tau h^{2(k-1+\delta _{k,2})}\int_{t_{n-1}}^{t_n}\norm{u_{ss}}_{\m}^2ds+\varepsilon _{21}\trb{\theta e^{n-1}+(1-\theta)e^{n}}^2,\\
    \\
    &~\sumT \langle (w_{2,T}^n)_t,\theta e_{2,T}^{n-1}+(1-\theta)e_{2,T}^{n}\rangle _{\partial T}\\
    \leq &~\frac{1}{4\varepsilon _{22}}Ch^{2(k-1+\delta _{k,2})}\norm{u_t(t_n)}_{\m}^2+\varepsilon _{22}\trb{\theta e^{n-1}+(1-\theta)e^{n}}^2,\\
    \\
    &~\sumT\langle w_{2,T}^n,e_{2,T}^n\rangle _{\partial T}\leq Ch^{k-1+\delta _{k,2}}\norm{u(t_n)}_{\m}\trb{e^n}\\
    \\
    &~\sumT ((w_{3,T}^n)_t-\overline\partial w_{3,T}^n,\theta e_{3,T}^{n-1}+(1-\theta)e_{3,T}^{n})_{T}\\
    \leq &~\frac{1}{4\varepsilon _{31}}C\tau h^{2(k-1)}\int_{t_{n-1}}^{t_n}\norm{u_{ss}}_{k+1}^2ds+\varepsilon _{31}\trb{\theta e^{n-1}+(1-\theta)e^{n}}^2,\\
    \\
    &~\sumT ((w_{3,T}^n)_t,\theta e_{3,T}^{n-1}+(1-\theta)e_{3,T}^{n})_{T}\\
    \leq &~\frac{1}{4\varepsilon _{32}}Ch^{2(k-1)}\norm{u_t(t_n)}_{k+1}^2+\varepsilon _{32}\trb{\theta e^{n-1}+(1-\theta)e^{n}}^2,\\
    \\
    &~\sumT (w_{3,T}^n,e_{3,T}^n) _{T}\leq Ch^{k-1}\norm{u(t_n)}_{k+1}\trb{e^n}\\
  \end{align*}

  Combining the above inequalities, (\ref{left}) and (\ref{right}), we get
  \begin{align*}
    &~\frac{1}{\tau}\norm{e^n-e^{n-1}}^2+(\theta -\frac{1}{2})\trb{e^n-e^{n-1}}^2+\frac{1}{2}\Big(\trb{e^n}^2-\trb{e^{n-1}}^2\Big)\\
    \leq &~\frac{1}{4\varepsilon}\norm{\theta u_t(t_n)+(1-\theta)u_t(t_{n-1})-\overline{\partial}u(t_n)}^2+\varepsilon \norm{e^n-e^{n-1}}^2\\
    &~+\tau\sumT \Big[\overline\partial\langle w_{1,T}^n,e_{1,T}^n\rangle _{\partial T}-\overline\partial\langle w_{2,T}^n,e_{2,T}^n\rangle _{\partial T}+\overline\partial (w_{3,T}^n,e_{3,T}^n)_{T}\Big]\\
    &~+(\frac{1}{\varepsilon _{11}}+\frac{1}{\varepsilon _{21}}+\frac{1}{\varepsilon _{31}})C\tau ^2h^{2(k-1)}\int_{t_{n-1}}^{t_n}\norm{u_{ss}}_{\m}^2ds\\
    &~+(\frac{1}{\varepsilon _{12}}+\frac{1}{\varepsilon _{22}}+\frac{1}{\varepsilon _{32}})C\tau h^{2(k-1)}\norm{u_t(t_n)}_{\m}^2\\
    &~+(\varepsilon _{11}+\varepsilon _{21}+\varepsilon _{31}+\varepsilon _{12}+\varepsilon _{22}+\varepsilon _{32})\tau\trb{\theta e^{n-1}+(1-\theta)e^{n}}^2.
  \end{align*}
  Choosing $\varepsilon =\frac{1}{\tau}$ and $\varepsilon _{11}+\varepsilon _{21}+\varepsilon _{31}+\varepsilon _{12}+\varepsilon _{22}+\varepsilon _{32}=1$, we obtain
\begin{align*}
  &~(\theta -\frac{1}{2})\trb{e^n-e^{n-1}}^2+\frac{1}{2}\Big(\trb{e^n}^2-\trb{e^{n-1}}^2\Big)\\
  \leq &~\frac{\tau}{4}\norm{\theta u_t(t_n)+(1-\theta)u_t(t_{n-1})-\overline{\partial}u(t_n)}^2+\tau\sumT \Big[\overline\partial\langle w_{1,T}^n,e_{1,T}^n\rangle _{\partial T}-\overline\partial\langle w_{2,T}^n,e_{2,T}^n\rangle _{\partial T}+\overline\partial (w_{3,T}^n,e_{3,T}^n)_{T}\Big]\\
  &~+C\tau ^2h^{2(k-1)}\int_{t_{n-1}}^{t_n}\norm{u_{ss}}_{\m}^2ds+C\tau h^{2(k-1)}\norm{u_t(t_n)}_{\m}^2+\tau\trb{\theta e^{n-1}+(1-\theta)e^{n}}^2,
\end{align*}
  which implies
\begin{align*}
  \frac{1}{2}\trb{e^n}^2
  \leq &~\frac{1}{2}\trb{e^{n-1}}^2+\frac{\tau}{4}\norm{\theta u_t(t_n)+(1-\theta)u_t(t_{n-1})-\overline{\partial}u(t_n)}^2\\
  &~+\tau\sumT \Big[\overline\partial\langle w_{1,T}^n,e_{1,T}^n\rangle _{\partial T}-\overline\partial\langle w_{2,T}^n,e_{2,T}^n\rangle _{\partial T}+\overline\partial (w_{3,T}^n,e_{3,T}^n)_{T}\Big]\\
  &~+C\tau ^2h^{2(k-1)}\int_{t_{n-1}}^{t_n}\norm{u_{ss}}_{\m}^2ds+C\tau h^{2(k-1)}\norm{u_t(t_n)}_{\m}^2+\tau\trb{\theta e^{n-1}+(1-\theta)e^{n}}^2\\
  \leq &~\frac{1}{2}\trb{e^{0}}^2+\frac{\tau}{4}\sum _{j=1}^n\norm{\theta u_t(t_j)+(1-\theta)u_t(t_{j-1})-\overline{\partial}u(t_j)}^2\\
  &~+\sumT \Big[\langle w_{1,T}^n,e_{1,T}^n\rangle _{\partial T}-\langle w_{1,T}^0,e_{1,T}^0\rangle _{\partial T}-\langle w_{2,T}^n,e_{2,T}^n\rangle _{\partial T}+\langle w_{2,T}^0,e_{2,T}^0\rangle _{\partial T}+(w_{3,T}^n,e_{3,T}^n)_{T}-(w_{3,T}^0,e_{3,T}^0)_{T}\Big]\\
  &~+C\tau ^2h^{2(k-1)}\int_{0}^{t_n}\norm{u_{ss}}_{\m}^2ds+C\tau h^{2(k-1)}\sum _{j=1}^n\norm{u_t(t_j)}_{\m}^2+\tau\sum _{j=1}^n\trb{\theta e^{j-1}+(1-\theta)e^{j}}^2\\
  \leq &~\frac{1}{2}\trb{e^{0}}^2+\frac{\tau}{4}\sum _{j=1}^n\norm{\theta u_t(t_j)+(1-\theta)u_t(t_{j-1})-\overline{\partial}u(t_j)}^2\\
  &~+Ch^{k-1}\norm{u(t_n)}_{\m}\trb{e^n}+Ch^{k-1}\norm{u(0)}_{\m}\trb{e^0}\\
  &~+C\tau ^2h^{2(k-1)}\int_{0}^{t_n}\norm{u_{ss}}_{\m}^2ds+C\tau h^{2(k-1)}n\norm{u_t}_{\m,\infty}^2+\tau\sum _{j=1}^n\trb{\theta e^{j-1}+(1-\theta)e^{j}}^2\\
  \leq &~\frac{1}{2}\trb{e^{0}}^2+\frac{\tau}{4}\sum _{j=1}^n\norm{\theta u_t(t_j)+(1-\theta)u_t(t_{j-1})-\overline{\partial}u(t_j)}^2\\
  &~+Ch^{2(k-1)}\norm{u(t_n)}_{\m}^2+\frac{1}{4}\trb{e^n}^2+Ch^{2(k-1)}\norm{u(0)}_{\m}^2+\frac{1}{4}\trb{e^0}^2\\
  &~+C\tau ^2h^{2(k-1)}\int_{0}^{t_n}\norm{u_{ss}}_{\m}^2ds+Ch^{2(k-1)}\norm{u_t}_{\m,\infty}^2+\tau\sum _{j=1}^n\trb{\theta e^{j-1}+(1-\theta)e^{j}}^2.
\end{align*}
Therefore, we get
\begin{align*}
\frac{1}{4}\trb{e^n}^2
\leq &~\frac{3}{4}\trb{e^{0}}^2+\frac{\tau}{4}\sum _{j=1}^n\norm{\theta u_t(t_j)+(1-\theta)u_t(t_{j-1})-\overline{\partial}u(t_j)}^2\\
  &~+Ch^{2(k-1)}\norm{u}_{\m,\infty}^2+Ch^{2(k-1)}\norm{u(0)}_{\m}^2\\
  &~+C\tau ^2h^{2(k-1)}\int_{0}^{t_n}\norm{u_{ss}}_{\m}^2ds+Ch^{2(k-1)}\norm{u_t}_{\m,\infty}^2+\tau\sum _{j=1}^n\trb{\theta e^{j-1}+(1-\theta)e^{j}}^2\\
  \leq &~\frac{3}{4}\trb{e^{0}}^2+Ch^{2(k-1)}\Big(\norm{u(0)}_{\m}^2+\norm{u}_{\m,\infty}^2+\norm{u_t}_{\m,\infty}^2+\int_{0}^{t_n}\norm{u_{ss}}_{\m}^2ds\Big)\\
  &~+\frac{\tau}{4}\sum _{j=1}^n\norm{\theta u_t(t_j)+(1-\theta)u_t(t_{j-1})-\overline{\partial}u(t_j)}^2+\tau\sum _{j=1}^n\trb{\theta e^{j-1}+(1-\theta)e^{j}}^2
\end{align*}

When $\theta =1$, using (\ref{part_1}) and (\ref{err-equ-full-L2}), we have 
\begin{align*}
  \trb{e^n}^2
  \leq &~C\Big(\norm{e^0}^2+\trb{e^0}^2\Big)+Ch^{2(k-1)}\Big(\norm{u(0)}_{\m}^2+\norm{u}_{\m,\infty}^2+\norm{u_t}_{\m,\infty}^2\\
  &~+\int_{0}^{t_n}\norm{u_{ss}}_{\m}^2ds\Big)+\tau^2 C\int _{0}^{t_n}\norm{u_{ss}}^2ds.
\end{align*}

When $\theta =\frac{1}{2}$, by (\ref{part_2}) and (\ref{err-equ-full-L2-2}), we obtain
\begin{align*}
  \trb{e^n}^2
  \leq &~C\Big(\norm{e^0}^2+\trb{e^0}^2\Big)+Ch^{2(k-1)}\Big(\norm{u(0)}_{\m}^2+\norm{u}_{\m,\infty}^2+\norm{u_t}_{\m,\infty}^2\\
  &~+\int_{0}^{t_n}\norm{u_{ss}}_{\m}^2ds\Big)+\tau^4 C\int _{0}^{t_n}\norm{u_{sss}}^2ds.
\end{align*}
\end{proof}

Since the difference of $\tau\sum _{j=1}^n\trb{\theta e^{j-1}+(1-\theta)e^{j}}^2$ and $\tau\sum _{j=1}^n\trb{\theta e^{j}+(1-\theta)e^{j-1}}^2$, we cannot get the estimation of $\trb{e^n}$ when $\theta\in (\frac{1}{2},1)$. And we can observe that the $L^2$ norm of the error $e^n$ does not reach the optimal convergence order. In order to obtain the optimal error estimates, we first introduce an elliptic projection $E_h$. For any $u\in H_0^2(\Omega)$, we define $E_hu\in V_h$ such that
\begin{align}\label{Eh-def}
  (\WL E_hu,\WL v)=(\Delta ^2u,v),\qquad \forall v\in V_h^0.
\end{align}

In fact, $E_hu$ is the SFWG numerical solution of the following equation:
\begin{align}\label{eq}
  \left\{
\begin{array}{rcl}
\Delta ^2 u &=& f,\qquad \text{in }\Omega,\\
u &=& 0,\qquad \text{on }\partial\Omega,\\
\frac{\partial u}{\partial \bn} &=&0,\qquad \text{on }\partial\Omega.
\end{array}
\right.
\end{align}

And we have the following estimates for corresponding error:
\begin{lemma}\cite{BiharmonicSFWG}\label{Eh-ests}
  For $u\in H^{\max{\{k+1,4\}}}(\Omega)$, $k\geq 2$, $\alpha =\min{\{k,3\}}$, we have the error estimates as follows:
  \begin{align}
    \label{Eh-trb}\trb{u-E_hu}&\leq Ch^{k-1}\norm{u}_{\m},\\
    \label{Eh-L2}\norm{u-E_hu}&\leq Ch^{k+\alpha -2}\norm{u}_{\m}.
  \end{align}
\end{lemma}

  

\begin{lemma}\cite{BiharmonicSFWG}\label{Qh_u-ests}
  For $u\in H^{\max{\{k+1,4\}}}(\Omega)$, $k\geq 2$, we have
  \begin{align}
    \trb{u-Q_hu}&\leq Ch^{k-1}\norm{u}_{k+1},\\
    \norm{u-Q_hu}&\leq Ch^{k+1}\norm{u}_{k+1}.
  \end{align}
\end{lemma}

Combining Lemma \ref{Eh-ests} with Lemma \ref{Qh_u-ests}, we get
\begin{align}
  \trb{E_hu-Q_hu}&\leq Ch^{k-1}\norm{u}_{\m},\\
  \norm{E_hu-Q_hu}&\leq Ch^{k+\alpha -2}\norm{u}_{\m}.
\end{align}

Let $u_h(t)-Q_hu(t)=u_h-E_hu+E_hu-Q_hu:=\theta (t)+\rho (t)$, and we have the following estimates for the semi-discrete scheme (\ref{semi-scheme}).
\begin{theorem}\label{semi-err}
  Let $u\in H^{\max{\{k+1,4\}}}(\Omega)$ be the exact solution of (\ref{problem-eq}), $u_h$ be the numerical solution of semi-discrete scheme (\ref{semi-scheme}) and we have
  \begin{align}
    \label{uh-L2}\norm{u_h-Q_hu}\leq \ &\norm{u_h(0)-Q_hu(0)}+Ch^{k+\alpha -2}\Big(\norm{\psi}_{\m}+\Big(\int _0^t \norm{u_s}^2_{\m}ds\Big)^\frac{1}{2}\Big),\\
    \label{uh-trb}\trb{u_h(t)-Q_hu(t)}
    \leq \ & \trb{u_h(0)-Q_hu(0)}+Ch^{k-1}(\norm{\psi}_{\m}+\norm{u}_{\m})\\
    &+Ch^{k+\alpha -2}\Big(\int _0^t\norm{u_s}_{\m}^2\Big)^\frac{1}{2}.\nonumber
  \end{align}
  where $k\geq 2$, $\alpha = \min{\{k,3\}}$.
\end{theorem}

\begin{proof}
  Due to $u_h(t)-Q_hu(t)=\theta (t)+\rho (t)$, we only derive the estimates of $\theta (t)$ and $\rho (t)$ respectively.

  For $\rho (t)$, we have
  \begin{align*}
    \norm{\rho (t)}^2 &\leq Ch^{2(k+\alpha -2)}\norm{u}_{\m}^2\\
    &\leq Ch^{2(k+\alpha -2)}\Big(\norm{\psi}_{\m}^2+\int _0^t \norm{u_s}^2_{\m}ds\Big),
  \end{align*}
  which implies
  \begin{align}\label{rho-L2}
    \norm{\rho (t)}\leq Ch^{k+\alpha -2}\Big(\norm{\psi}_{\m}+\Big(\int _0^t \norm{u_s}^2_{\m}ds\Big)^\frac{1}{2}\Big).
  \end{align}

  And from (\ref{Eh-trb}), we know
  \begin{align}\label{rho-trb}
    \trb{\rho (t)}\leq Ch^{k-1}\norm{u}_{\m}.
  \end{align}
  By (\ref{semi-scheme}), (\ref{Eh-def}) and (\ref{eq}), we have
  \begin{align*}
    (\theta _t,\chi )+(\WL \theta,\WL \chi )
    &=(u_{h,t},\chi )+(\WL u_h,\WL\chi )-(E_hu_t,\chi )-(\WL E_hu,\WL\chi )\\
    &=(f,\chi)-(E_hu_t,\chi)-(\Delta ^2u,\chi)\\
    &=(u_t,\chi)-(E_hu_t,\chi)\\
    &=(Q_hu_t,\chi)-(E_hu_t,\chi)\\
    &=-(\rho _t,\chi),
  \end{align*}
  for any $\chi\in V_h^0$ holds true.

  Let $\chi =\theta$, we get
  \begin{align*}
    (\theta _t,\theta)+\trb{\theta}^2= -(\rho _t,\theta).
  \end{align*}
  Since $(\theta _t,\theta)=\frac{1}{2}\frac{d}{dt}\norm{\theta}^2=\norm{\theta}\frac{d}{dt}\norm{\theta}$ and $\trb{\theta}^2$ is non-negative, we can obtain
  \begin{align*}
    \norm{\theta}\frac{d}{dt}\norm{\theta}&\leq \norm{\rho _t}\norm \theta,
  \end{align*}
  that is
  \begin{align*}
    \frac{d}{dt}\norm{\theta}&\leq \norm{\rho _t}.
  \end{align*}
  Further, we have
  \begin{align*}
    \norm{\theta}\leq\norm{\theta (0)}+\int_0^t\norm{\rho _s}ds.
  \end{align*}
  Due to
  \begin{align*}
    \norm{\theta (0)} &=\norm{u_h(0)-E_hu(0)}\\
    &\leq \norm{u_h(0)-Q_hu(0)}+\norm{Q_hu(0)-E_hu(0)}\\
    &\leq \norm{u_h(0)-Q_hu(0)}+Ch^{k+\alpha -2}\norm{u(0)}_{\m},
  \end{align*}
  and $\norm{\rho _t}=\norm{E_hu_t-Q_hu_t}\leq Ch^{k+\alpha -2}\norm{u_t}_{\m}$, we have
  \begin{align}\label{theta-L2}
    \norm{\theta (t)}\leq\norm{u_h(0)-Q_hu(0)}+Ch^{k+\alpha -2}\Big(\norm{\psi}_{\m}+\int _0^t\norm{u_s}_{\m}ds\Big).
  \end{align}
  Combining (\ref{theta-L2}) with (\ref{rho-L2}), we get
  \begin{align*}
    \norm{u_h(t)-Q_hu(t)}
    &\leq \norm{\theta (t)}+\norm{\rho (t)}\\
    &\leq \norm{u_h(0)-Q_hu(0)}+Ch^{k+\alpha -2}\Big(\norm{\psi}_{\m}+\Big(\int _0^t\norm{u_s}_{\m}^2ds\Big)^\frac{1}{2}\Big),
  \end{align*}
  which obtains (\ref{uh-L2}).

  If we choose $\chi =\theta _t$, we have
  \begin{align*}
    \norm{\theta _t}^2+(\WL \theta,\WL\theta _t)&=-(\rho _t,\theta _t)\\
    &\leq \frac{1}{2}\norm{\rho _t}^2+\frac{1}{2}\norm{\theta _t}^2,
  \end{align*}
  which means
  \begin{align*}
    \frac{1}{2}\norm{\theta _t}^2+\frac{1}{2}\frac{d}{dt}\trb{\theta}^2\leq\frac{1}{2}\norm{\rho _t}^2.
  \end{align*}
  
  Further, we know
  \begin{align*}
    \frac{1}{2}\frac{d}{dt}\trb{\theta}^2\leq\frac{1}{2}\norm{\rho _t}^2.
  \end{align*}
  Then, we have
  \begin{align*}
    \trb{\theta(t)}^2 \leq \trb{\theta (0)}^2+\int _0^t\norm{\rho _s}^2ds.
  \end{align*}

  Due to
  \begin{align*}
    \trb{\theta (0)}&=\trb{u_h(0)-E_hu(0)}\\
    &\leq\trb{u_h(0)-Q_hu(0)}+\trb{Q_hu(0)-E_hu(0)}\\
    &\leq\trb{u_h(0)-Q_hu(0)}+Ch^{k-1}\norm{u(0)}_{\m},
  \end{align*}
  and $\norm{\rho _t}=\norm{E_hu_t-Q_hu_t}\leq Ch^{k+\alpha -2}\norm{u_t}_{\m}$, we get
  \begin{align*}
    \trb{\theta (t)}^2\leq(\trb{u_h(0)-Q_hu(0)}+Ch^{k-1}\norm{\psi}_{\m})^2+Ch^{2(k+\alpha -2)}\int _0^t\norm{u_s}_{\m}^2ds,
  \end{align*}
  which implies
  \begin{align}\label{theta-trb}
    \trb{\theta (t)}\leq \trb{u_h(0)-Q_hu(0)}+Ch^{k-1}\norm{\psi}_{\m}+Ch^{k+\alpha -2}\Big(\int _0^t\norm{u_s}_{\m}^2ds\Big)^\frac{1}{2}.
  \end{align}

  Using (\ref{theta-trb}) and (\ref{rho-trb}), we have
  \begin{align*}
    \trb{u_h(t)-Q_hu(t)}\leq\ &\trb{\theta (t)}+\trb{\rho (t)}\\
    \leq \ & \trb{u_h(0)-Q_hu(0)}+Ch^{k-1}(\norm{\psi}_{\m}+\norm{u}_{\m})\\
    &+Ch^{k+\alpha -2}\Big(\int _0^t\norm{u_s}_{\m}^2ds\Big)^\frac{1}{2}.
  \end{align*}
  The proof is completed.
\end{proof}

For the full-discrete scheme (\ref{full-scheme}), we have the following estimates.

\begin{theorem}\label{full-est}
  Let $u\in H^{\max{\{k+1,4\}}}(\Omega)$ $(k\geq 2$, $\alpha = \min{\{k,3\}})$ be the exact solution of (\ref{problem-eq}) and $U^n$ be the numerical solution of full-discrete scheme (\ref{full-scheme}), then we have the following estimates.
  When $\theta\in(\frac{1}{2},1]$,
  \begin{align}
    \label{full-L2}\norm{U^n-Q_hu(t_n)}
    \leq &\norm{U^0-Q_hu(0)}+Ch^{k+\alpha -2}\Big(\norm{\psi}_{\m}+\Big(\int _0^{t_n}\norm{u_s}_{\m}^2ds\Big)^\frac{1}{2}\Big)\\ \nonumber
    &+C\tau\Big(\int _0^{t_n}\norm{u_{ss}}^2ds\Big)^\frac{1}{2},\\ 
    \label{full-trb}\trb{U^n-Q_hu(t_n)}
    \leq &\trb{U^0-Q_hu(0)}+Ch^{k-1}\Big(\norm{\psi}_{\m}+\Big(\int _0^{t_n}\norm{u_s}_{\m}^2ds\Big)^\frac{1}{2}\Big)\\ \nonumber
    &+C\tau\Big(\int _0^{t_n}\norm{u_{ss}}^2ds\Big)^\frac{1}{2}.
  \end{align}
  When $\theta =\frac{1}{2}$,
  \begin{align}
    \label{full-L2-2}\norm{U^n-Q_hu(t_n)}
    \leq &\norm{U^0-Q_hu(0)}+Ch^{k+\alpha -2}\Big(\norm{\psi}_{\m}+\Big(\int _0^{t_n}\norm{u_s}_{\m}^2ds\Big)^\frac{1}{2}\Big)\\ \nonumber
    &+C\tau ^2\Big(\int _0^{t_n}\norm{u_{sss}}^2ds\Big)^\frac{1}{2},\\
    \label{full-trb-2}\trb{U^n-Q_hu(t_n)}
    \leq &\trb{U^0-Q_hu(0)}+Ch^{k-1}\Big(\norm{\psi}_{\m}+\Big(\int _0^{t_n}\norm{u_s}_{\m}^2ds\Big)^\frac{1}{2}\Big)\\ \nonumber
    &+C\tau ^2\Big(\int _0^{t_n}\norm{u_{sss}}^2ds\Big)^\frac{1}{2}.
  \end{align}
\end{theorem}

\begin{proof}
  Let $U^n-Q_hu(t_n)=(U^n-E_hu(t_n))+(E_hu(t_n)-Q_hu(t_n)):=\eta ^n+\rho ^n$, and we know
  \begin{align*}
    \norm{\rho ^n}=&\norm{E_hu(t_n)-Q_hu(t_n)}\\
    \leq& Ch^{k+\alpha -2}\norm{u(t_n)}_{\m}\\
    \leq& Ch^{k+\alpha -2}\Big(\norm{\psi}_{\m}+\Big(\int _0^{t_n}\norm{u_s}_{\m}^2ds\Big)^\frac{1}{2}\Big),\\
    \trb{\rho ^n}=&\trb{E_hu(t_n)-Q_hu(t_n)}\\
    \leq& Ch^{k-1}\norm{u(t_n)}_{\m}\\
    \leq& Ch^{k-1}\Big(\norm{\psi}_{\m}+\Big(\int _0^{t_n}\norm{u_s}_{\m}^2ds\Big)^\frac{1}{2}\Big).
  \end{align*}
  Therefore, we only need to estimate $\eta ^n$.

  Using the full-discrete scheme (\ref{full-scheme}), the definition of $E_h$ and the equation (\ref{eq}), we have
  \begin{align*}
    &(\overline{\partial}\eta ^n,v_0)+(\WL (\theta \eta ^n+(1-\theta)\eta ^{n-1}),\WL v)\\
    =&(\OP U^n,v_0)-(\OP E_hu(t_n),v_0)+(\WL (\theta U^n+(1-\theta)U^{n-1}),\WL v)\\
    &-(\WL (\theta E_hu(t_n)+(1-\theta)E_hu(t_{n-1})),\WL v)\\
    =&(\theta f(t_n)+(1-\theta)f(t_{n-1}),v_0)-(\OP E_hu(t_n),v_0)-(\Delta ^2(\theta u(t_n)+(1-\theta)u(t_{n-1})),v)\\
    =&-(\OP E_hu(t_n),v_0)+(\theta u_t(t_n)+(1-\theta)u_t(t_{n-1}),v_0)\\
    =&-(\OP \rho ^n,v_0)-(\OP Q_hu(t_n)-(\theta u_t(t_n)+(1-\theta)u_t(t_{n-1})),v_0)\\
    =&-(\OP \rho ^n,v_0)-(\OP u(t_n)-(\theta u_t(t_n)+(1-\theta)u_t(t_{n-1})),v_0),
  \end{align*}
  for any $v\in V_h^0$ hold true.

  Choosing $v=\theta\eta ^n+(1-\theta)\eta ^{n-1}$, it follows that
  \begin{align*}
    &\frac{1}{2\tau}\norm{\eta ^n}^2-\frac{1}{2\tau}\norm{\eta ^{n-1}}^2+\frac{1}{\tau}(\theta -\frac{1}{2})\norm{\eta ^n-\eta ^{n-1}}^2+\trb{\theta \eta ^n+(1-\theta)\eta ^{n-1}}^2\\
    =&-(\OP \rho ^n,\theta\eta ^n+(1-\theta)\eta ^{n-1})-(\OP u(t_n)-(\theta u_t(t_n)+(1-\theta)u_t(t_{n-1})),\theta\eta ^n+(1-\theta)\eta ^{n-1})\\
    \leq &(\norm{\OP \rho ^n}+\norm{\OP u(t_n)-(\theta u_t(t_n)+(1-\theta)u_t(t_{n-1}))})\norm{\theta\eta ^n+(1-\theta)\eta ^{n-1}},
  \end{align*}

  Using (\ref{part_1}), we get
  \begin{align*}
    &\norm{\eta ^n}^2-\norm{\eta ^{n-1}}^2+2(\theta -\frac{1}{2})\norm{\eta ^n-\eta ^{n-1}}^2+2\tau\trb{\theta \eta ^n+(1-\theta)\eta ^{n-1}}^2\\
    \leq &2\tau(\norm{\OP \rho ^n}+\norm{\OP u(t_n)-(\theta u_t(t_n)+(1-\theta)u_t(t_{n-1}))})\norm{\theta\eta ^n+(1-\theta)\eta ^{n-1}}\\
    \leq &2\tau\norm{\OP \rho ^n}\norm{\theta\eta ^n+(1-\theta)\eta ^{n-1}}+2\tau ^\frac{3}{2}\Big(\int _{t_{n-1}}^{t_n}\norm{u_{ss}}^2ds\Big)^{\frac{1}{2}}\norm{\theta\eta ^n+(1-\theta)\eta ^{n-1}}\\
    \leq &\frac{\tau}{\varepsilon}\norm{\OP\rho ^n}^2+\frac{\varepsilon}{4}4\tau\norm{\theta\eta ^n+(1-\theta)\eta ^{n-1}}^2+\frac{1}{\varepsilon}\tau ^2\int _{t_{n-1}}^{t_n}\norm{u_{ss}}^2ds+\frac{\varepsilon}{4}4\tau\norm{\theta\eta ^n+(1-\theta)\eta ^{n-1}}^2\\
    \leq &\frac{\tau}{\varepsilon}\norm{\OP\rho ^n}^2+\frac{\tau ^2}{\varepsilon}\int _{t_{n-1}}^{t_n}\norm{u_{ss}}^2ds+2\varepsilon\tau\norm{\theta\eta ^n+(1-\theta)\eta ^{n-1}}^2.
  \end{align*}

  Let $\varepsilon =\frac{1}{4}$, and we have
  \begin{align*}
    \norm{\eta ^n}^2-\norm{\eta ^{n-1}}^2
    \leq 4\tau\norm{\OP \rho ^n}^2+4\tau ^2\int _{t_{n-1}}^{t_n}\norm{u_{ss}}^2ds,
  \end{align*}
  where we use the inequality (\ref{L2-trb}).

  Since
  \begin{align}
    \label{part_rho}\norm{\OP \rho ^n}^2=&\int _\Omega \Big(\frac{E_hu(t_n)-Q_hu(t_n)-(E_hu(t_{n-1})-Q_hu(t_{n-1}))}{\tau}\Big)^2dx\\ \nonumber
    =&\frac{1}{\tau ^2}\int _\Omega\Big(\int _{t_{n-1}}^{t_n}(E_hu_s-Q_hu_s)ds\Big)^2dx\\ \nonumber
    \leq &\frac{1}{\tau ^2}\int _\Omega\Big(\int _{t_{n-1}}^{t_n}(E_hu_s-Q_hu_s)^2ds\Big)\Big(\int_{t_{n-1}}^{t_n}1ds\Big)dx\\ \nonumber
    \leq &\frac{1}{\tau}\int _{t_{n-1}}^{t_n}\norm{E_hu_s-Q_hu_s}^2ds\\ \nonumber
    \leq &\frac{1}{\tau}Ch^{2(k+\alpha -2)}\int _{t_{n-1}}^{t_n}\norm{u_s}_{\m}^2ds, \nonumber
  \end{align}
  we get 
  \begin{align*}
    \norm{\eta ^n}^2\leq \norm{\eta ^{n-1}}^2+Ch^{2(k+\alpha -2)}\int _{t_{n-1}}^{t_n}\norm{u_s}_{\m}^2ds+C\tau ^2\int _{t_{n-1}}^{t_n}\norm{u_{ss}}^2ds.
  \end{align*}

  It follows that 
  \begin{align*}
    \norm{\eta ^n}^2\leq \norm{\eta ^{0}}^2+Ch^{2(k+\alpha -2)}\int _{0}^{t_n}\norm{u_s}_{\m}^2ds+C\tau ^2\int _{0}^{t_n}\norm{u_{ss}}^2ds.
  \end{align*}

  And with 
  \begin{align*}
    \norm{\eta ^0}=\norm{U^0-E_hu(0)}\leq &~\norm{U^0-Q_hu(0)}+\norm{Q_hu(0)-E_hu(0)}\\
    \leq &~\norm{U^0-Q_hu(0)}+Ch^{k+\alpha -2}\norm{\psi}_{\m},
  \end{align*}
  we obtain
  \begin{align*}
    \norm{\eta ^n}\leq &~\norm{U^0-Q_hu(0)}+Ch^{k+\alpha -2}\norm{\psi}_{\m}+Ch^{k+\alpha -2}\Big(\int _{0}^{t_n}\norm{u_s}_{\m}^2ds\Big)^\frac{1}{2}\\
    &~+C\tau \Big(\int _{0}^{t_n}\norm{u_{ss}}^2ds\Big)^\frac{1}{2}.
  \end{align*}

  Using $U^n-Q_hu(t_n)=\eta ^n+\rho ^n$, we have (\ref{full-L2}).

  By (\ref{part_2}), we know
  \begin{align*}
    &\norm{\eta ^n}^2-\norm{\eta ^{n-1}}^2+2(\theta -\frac{1}{2})\norm{\eta ^n-\eta ^{n-1}}^2+2\tau\trb{\theta \eta ^n+(1-\theta)\eta ^{n-1}}^2\\
    \leq &2\tau(\norm{\OP \rho ^n}+\norm{\OP u(t_n)-(\theta u_t(t_n)+(1-\theta)u_t(t_{n-1}))})\norm{\theta\eta ^n+(1-\theta)\eta ^{n-1}}\\
    \leq &2\tau\norm{\OP \rho ^n}\norm{\theta\eta ^n+(1-\theta)\eta ^{n-1}}+\tau ^\frac{5}{2}\Big(\int _{t_{n-1}}^{t_n}\norm{u_{sss}}^2ds\Big)^{\frac{1}{2}}\norm{\theta\eta ^n+(1-\theta)\eta ^{n-1}}\\
    \leq &\frac{\tau}{\varepsilon}\norm{\OP\rho ^n}^2+\frac{\varepsilon}{4}4\tau\norm{\theta\eta ^n+(1-\theta)\eta ^{n-1}}^2+\frac{1}{\varepsilon}\frac{\tau ^4}{4}\int _{t_{n-1}}^{t_n}\norm{u_{sss}}^2ds+\frac{\varepsilon}{4}4\tau\norm{\theta\eta ^n+(1-\theta)\eta ^{n-1}}^2\\
    \leq &\frac{\tau}{\varepsilon}\norm{\OP\rho ^n}^2+\frac{\tau ^4}{4\varepsilon}\int _{t_{n-1}}^{t_n}\norm{u_{sss}}^2ds+2\varepsilon\tau\norm{\theta\eta ^n+(1-\theta)\eta ^{n-1}}^2,
  \end{align*}

  Next, similar to the proof for the case $\theta\in(\frac{1}{2},1]$, we can get (\ref{full-L2-2}).

  For the equation
  \begin{align*}
    &(\overline{\partial}\eta ^n,v_0)+(\WL (\theta \eta ^n+(1-\theta)\eta ^{n-1}),\WL v)\\
    =&-(\OP \rho ^n,v_0)-(\OP u(t_n)-(\theta u_t(t_n)+(1-\theta)u_t(t_{n-1})),v_0),\qquad \forall v\in V_h^0,
  \end{align*}
  we set $v=\eta ^n-\eta ^{n-1}$ and arrive at 
  \begin{align*}
    &(\overline{\partial}\eta ^n,\eta ^n-\eta ^{n-1})+(\WL (\theta \eta ^n+(1-\theta)\eta ^{n-1}),\WL (\eta ^n-\eta ^{n-1}))\\
    =&-(\OP \rho ^n,\eta ^n-\eta ^{n-1})-(\OP u(t_n)-(\theta u_t(t_n)+(1-\theta)u_t(t_{n-1})),\eta ^n-\eta ^{n-1}).
  \end{align*}

  From $\theta \eta ^n+(1-\theta)\eta ^{n-1}=(\theta -\frac{1}{2})(\eta ^n-\eta ^{n-1})+\frac{1}{2}(\eta ^n+\eta ^{n-1})$, the Cauchy-Schwarz inequality, and the Young inequality, we have
  \begin{align*}
    &\frac{1}{\tau}\norm{\eta ^n-\eta ^{n-1}}^2+(\theta -\frac{1}{2})\trb{\eta ^n-\eta ^{n-1}}^2+\frac{1}{2}(\trb{\eta ^n}^2-\trb{\eta ^{n-1}}^2)\\
    =&-(\OP \rho ^n,\eta ^n-\eta ^{n-1})-(\OP u(t_n)-(\theta u_t(t_n)+(1-\theta)u_t(t_{n-1})),\eta ^n-\eta ^{n-1})\\
    \leq &\norm{\OP \rho ^n}\norm{\eta ^n-\eta ^{n-1}}+\norm{\OP u(t_n)-(\theta u_t(t_n)+(1-\theta)u_t(t_{n-1}))}\norm{\eta ^n-\eta ^{n-1}}\\
    \leq &\frac{2\tau}{\varepsilon}\norm{\OP \rho ^n}^2+\frac{\varepsilon}{8\tau}\norm{\eta ^n-\eta ^{n-1}}^2+\frac{2\tau}{\varepsilon}\norm{\OP u(t_n)-(\theta u_t(t_n)+(1-\theta)u_t(t_{n-1}))}^2+\frac{\varepsilon}{8\tau}\norm{\eta ^n-\eta ^{n-1}}^2\\
    \leq &\frac{2\tau}{\varepsilon}\norm{\OP \rho ^n}^2+\frac{\varepsilon}{4\tau}\norm{\eta ^n-\eta ^{n-1}}^2+\frac{2\tau}{\varepsilon}\norm{\OP u(t_n)-(\theta u_t(t_n)+(1-\theta)u_t(t_{n-1}))}^2.
  \end{align*}

  When $\theta\in(\frac{1}{2},1]$, selecting $\varepsilon =\frac{1}{2}$ and using (\ref{part_1}), (\ref{part_rho}), we obtain
  \begin{align*}
    &\frac{1}{\tau}\norm{\eta ^n-\eta ^{n-1}}^2+(\theta -\frac{1}{2})\trb{\eta ^n-\eta ^{n-1}}^2+\frac{1}{2}(\trb{\eta ^n}^2-\trb{\eta ^{n-1}}^2)\\
    \leq & 4\tau\norm{\OP \rho ^n}^2+\frac{1}{8\tau}\norm{\eta ^n-\eta ^{n-1}}^2+4\tau ^2\int _{t_{n-1}}^{t_n}\norm{u_{ss}}^2ds\\
    \leq & Ch^{2(k+\alpha -2)}\int _{t_{n-1}}^{t_n}\norm{u_s}_{\m}^2ds+\frac{1}{8\tau}\norm{\eta ^n-\eta ^{n-1}}^2+4\tau ^2\int _{t_{n-1}}^{t_n}\norm{u_{ss}}^2ds,
  \end{align*}
  which show
  \begin{align*}
    &\frac{7}{8\tau}\norm{\eta ^n-\eta ^{n-1}}^2+(\theta -\frac{1}{2})\trb{\eta ^n-\eta ^{n-1}}^2+\frac{1}{2}(\trb{\eta ^n}^2-\trb{\eta ^{n-1}}^2)\\
    \leq & Ch^{2(k+\alpha -2)}\int _{t_{n-1}}^{t_n}\norm{u_s}_{\m}^2ds+4\tau ^2\int _{t_{n-1}}^{t_n}\norm{u_{ss}}^2ds.
  \end{align*}

  With $\norm{\eta ^n-\eta ^{n-1}}^2\geq 0$ and $\trb{\eta ^n-\eta ^{n-1}}^2\geq 0$, we get 
  \begin{align*}
    \trb{\eta ^n}^2\leq \trb{\eta ^{n-1}}^2+Ch^{2(k+\alpha -2)}\int _{t_{n-1}}^{t_n}\norm{u_s}_{\m}^2ds+4\tau ^2\int _{t_{n-1}}^{t_n}\norm{u_{ss}}^2ds,
  \end{align*}
  which implies
  \begin{align*}
    \trb{\eta ^n}^2\leq \trb{\eta ^{0}}^2+Ch^{2(k+\alpha -2)}\int _{0}^{t_n}\norm{u_s}_{\m}^2ds+4\tau ^2\int _{0}^{t_n}\norm{u_{ss}}^2ds.
  \end{align*}

  Since 
  \begin{align*}
    \trb{\eta ^0}=\trb{U^0-E_hu(0)}\leq \trb{U^0-Q_hu(0)}+\trb{Q_hu(0)-E_hu(0)}\leq \trb{U^0-Q_hu(0)}+Ch^{k-1}\norm{\psi}_{\m},
  \end{align*}
  we have
  \begin{align*}
    \trb{\eta ^n}\leq &~\trb{U^0-Q_hu(0)}+Ch^{k-1}\norm{\psi}_{\m}+Ch^{k+\alpha -2}\Big(\int _{0}^{t_n}\norm{u_s}_{\m}^2ds\Big)^\frac{1}{2}\\
    &~+C\tau \Big(\int _{0}^{t_n}\norm{u_{ss}}^2ds\Big)^\frac{1}{2}.
  \end{align*}

  Combining the estimates $\rho ^n$ and $\eta ^n$, we know
  \begin{align*}
    \trb{U^n-Q_hu(t_n)}\leq &\trb{\rho ^n}+\trb{\eta ^n}\\
    \leq &\trb{U^0-Q_hu(0)}+Ch^{k-1}\norm{\psi}_{\m}+Ch^{k-1}\Big(\int _{0}^{t_n}\norm{u_s}_{\m}^2ds\Big)^\frac{1}{2}\\
    &+Ch^{k+\alpha -2}\Big(\int _{0}^{t_n}\norm{u_s}_{\m}^2ds\Big)^\frac{1}{2}
    +C\tau \Big(\int _{0}^{t_n}\norm{u_{ss}}^2ds\Big)^\frac{1}{2}\\
    \leq &\trb{U^0-Q_hu(0)}+Ch^{k-1}\norm{\psi}_{\m}+Ch^{k-1}\Big(\int _{0}^{t_n}\norm{u_s}_{\m}^2ds\Big)^\frac{1}{2}\\
    &~+C\tau \Big(\int _{0}^{t_n}\norm{u_{ss}}^2ds\Big)^\frac{1}{2}.
  \end{align*}

  When $\theta =\frac{1}{2}$, using (\ref{part_2}) and the same process, we can get (\ref{full-trb-2}). So, the theorem has been proved.
\end{proof}

\section{Numerical Experiments}
In this section, we use two numerical examples to verify the efficiencies
of SFWG finite element method for the fourth order parabolic equation (\ref{problem-eq}).

We consider the problem (\ref{problem-eq}) on the unit square $\Omega=(0,1)^2$ and the time interval $(0,1]$. We choose the true solution is
\begin{align}\label{ex}
  u=\sin\Big(2\pi (t^2+1)+\frac{\pi}{2}\Big)\sin \Big(2\pi x+\frac{\pi}{2}\Big)\sin\Big(2\pi y+\frac{\pi}{2}\Big).
\end{align}

\begin{table}[htbp]
  \centering
  \caption{Error values and convergence rates for (\ref{ex}) on triangular meshes using $P_2$ and $\theta =1$}
    \begin{tabular}{c|c|c|c|c|c|c}
    \hline
    $n$ & $\trb{Q_hu -u _h} $ & Rate & $\norm{Q_h u-u_h}_{2,h}$& Rate & $\norm{Q_hu -u _h}$ & Rate \\
    \hline
    4     & 1.0411E+02 & ---    & 1.4853E+01 & ---    & 2.6906E-01 & --- \\
    8     & 5.6458E+01 & 0.88  & 5.8690E+00 & 1.34  & 9.8158E-02 & 1.45  \\
    16    & 2.8813E+01 & 0.97  & 1.9380E+00 & 1.60  & 2.7275E-02 & 1.85  \\
    32    & 1.4483E+01 & 0.99  & 7.2400E-01 & 1.42  & 7.0213E-03 & 1.96  \\
    64    & 7.2522E+00 & 1.00  & 3.2203E-01 & 1.17  & 1.7754E-03 & 1.98  \\
    128   & 3.6277E+00 & 1.00  & 1.5587E-01 & 1.05  & 4.5133E-04 & 1.98  \\
    \hline
    \end{tabular}%
  \label{ex_triP2}%
\end{table}%

\begin{table}[htbp]
  \centering
  \caption{Error values and convergence rates for (\ref{ex}) on triangular meshes using $P_3$ and $\theta =1$}
    \begin{tabular}{c|c|c|c|c|c|c}
    \hline
    $n$ & $\trb{Q_hu -u _h} $ & Rate & $\norm{Q_h u-u_h}_{2,h}$& Rate & $\norm{Q_hu -u _h}$ & Rate \\
    \hline
    2     & 1.5479E+02 & ---    & 2.0388E+01 & ---    & 3.2198E-01 & --- \\
    4     & 4.0274E+01 & 1.94  & 5.1257E+00 & 1.99  & 3.0504E-02 & 3.40  \\
    8     & 1.0578E+01 & 1.93  & 1.2560E+00 & 2.03  & 2.2763E-03 & 3.74  \\
    12    & 4.7431E+00 & 1.98  & 5.5063E-01 & 2.03  & 4.6935E-04 & 3.89  \\
    16    & 2.6762E+00 & 1.99  & 3.0745E-01 & 2.03  & 1.5102E-04 & 3.94  \\
    20    & 1.7152E+00 & 1.99  & 1.9596E-01 & 2.02  & 6.2420E-05 & 3.96  \\
    24    & 1.1920E+00 & 2.00  & 1.3575E-01 & 2.01  & 3.0310E-05 & 3.96  \\
    \hline
    \end{tabular}%
  \label{ex_triP3}%
\end{table}%

\begin{table}[htbp]
  \centering
  \caption{Error values and convergence rates for (\ref{ex}) on polygonal meshes using $P_3$ and $\theta =1$}
    \begin{tabular}{c|c|c|c|c|c|c}
      \hline
      $n$ & $\trb{Q_hu -u _h} $ & Rate & $\norm{Q_h u-u_h}_{2,h}$& Rate & $\norm{Q_hu -u _h}$ & Rate \\
      \hline
    2     & 2.0786E+02 & ---    & 1.7224E+01 & ---    & 2.8812E-01 & --- \\
    4     & 6.4374E+01 & 1.69  & 5.5939E+00 & 1.62  & 6.7346E-02 & 2.10  \\
    8     & 1.8154E+01 & 1.83  & 1.0413E+00 & 2.43  & 5.3369E-03 & 3.66  \\
    12    & 8.3421E+00 & 1.92  & 3.9287E-01 & 2.40  & 1.1227E-03 & 3.84  \\
    16    & 4.7626E+00 & 1.95  & 1.9940E-01 & 2.36  & 3.6485E-04 & 3.91  \\
    20    & 3.0734E+00 & 1.96  & 1.1891E-01 & 2.32  & 1.5155E-04 & 3.94  \\
    24    & 2.1455E+00 & 1.97  & 7.8441E-02 & 2.28  & 7.3723E-05 & 3.95  \\
    28    & 1.5820E+00 & 1.98  & 5.5427E-02 & 2.25  & 4.0056E-05 & 3.96  \\
    32    & 1.2144E+00 & 1.98  & 4.1159E-02 & 2.23  & 2.3610E-05 & 3.96  \\
    \hline
    \end{tabular}%
  \label{ex_polyP3}%
\end{table}%

When $\theta =1$, on the triangular meshes, we state the results as shown in Table \ref{ex_triP2} and Table \ref{ex_triP3}. We choose the $P_2$ weak Galerkin finite element, set $j=k+3$, $P=1000$ and obtain the convergence rates in $H^2$ and $L^2$ norms are of order $O(h)$ and $O(h^2)$ in the Table \ref{ex_triP2}. Table \ref{ex_triP3} shows the convergence rates are $O(h^2)$ and $O(h^4)$ in $H^2$ and $L^2$ norms respectively when $k=3$, $j=k+4$ and $P=40000$. By $\theta =1$ and $j=k+6$, we get the results on polygonal meshes as shown in Table \ref{ex_polyP3} using $P=50000$. And the rates in above tables are consistent with the theoretical orders.

\begin{table}[htbp]
  \centering
  \caption{Error values and convergence rates for (\ref{ex}) on triangular meshes using $P_2$ and $\theta =\frac{1}{2}$}
    \begin{tabular}{c|c|c|c|c|c|c}
    \hline
    $n$ & $\trb{Q_hu -u _h} $ & Rate & $\norm{Q_h u-u_h}_{2,h}$& Rate & $\norm{Q_hu -u _h}$ & Rate \\
    \hline
    4     & 1.0411E+02 & ---    & 1.4853E+01 & ---    & 2.6906E-01 & --- \\
    8     & 5.6458E+01 & 0.88  & 5.8690E+00 & 1.34  & 9.8158E-02 & 1.45  \\
    16    & 2.8813E+01 & 0.97  & 1.9380E+00 & 1.60  & 2.7275E-02 & 1.85  \\
    32    & 1.4483E+01 & 0.99  & 7.2400E-01 & 1.42  & 7.0213E-03 & 1.96  \\
    64    & 7.2522E+00 & 1.00  & 3.2203E-01 & 1.17  & 1.7754E-03 & 1.98  \\
    128   & 3.6277E+00 & 1.00  & 1.5587E-01 & 1.05  & 4.5146E-04 & 1.98  \\
    \hline
    \end{tabular}%
  \label{ex_triP2_2}%
\end{table}%

\begin{table}[htbp]
  \centering
  \caption{Error values and convergence rates for (\ref{ex}) on triangular meshes using $P_3$ and $\theta =\frac{1}{2}$}
    \begin{tabular}{c|c|c|c|c|c|c}
    \hline
    $n$ & $\trb{Q_hu -u _h} $ & Rate & $\norm{Q_h u-u_h}_{2,h}$& Rate & $\norm{Q_hu -u _h}$ & Rate \\
    \hline
    2     & 1.5479E+02 & ---    & 2.0388E+01 & ---    & 3.2198E-01 & --- \\
    4     & 4.0274E+01 & 1.94  & 5.1257E+00 & 1.99  & 3.0504E-02 & 3.40  \\
    8     & 1.0578E+01 & 1.93  & 1.2560E+00 & 2.03  & 2.2763E-03 & 3.74  \\
    12    & 4.7431E+00 & 1.98  & 5.5063E-01 & 2.03  & 4.6935E-04 & 3.89  \\
    16    & 2.6762E+00 & 1.99  & 3.0745E-01 & 2.03  & 1.5102E-04 & 3.94  \\
    20    & 1.7152E+00 & 1.99  & 1.9596E-01 & 2.02  & 6.2420E-05 & 3.96  \\
    24    & 1.1920E+00 & 2.00  & 1.3575E-01 & 2.01  & 3.0310E-05 & 3.96  \\
    \hline
    \end{tabular}%
  \label{ex_triP3_2}%
\end{table}%

\begin{table}[htbp]
  \centering
  \caption{Error values and convergence rates for (\ref{ex}) on polygonal meshes using $P_3$ and $\theta =\frac{1}{2}$}
    \begin{tabular}{c|c|c|c|c|c|c}
      \hline
      $n$ & $\trb{Q_hu -u _h} $ & Rate & $\norm{Q_h u-u_h}_{2,h}$& Rate & $\norm{Q_hu -u _h}$ & Rate \\
      \hline
    2     & 2.0786E+02 & ---    & 1.7224E+01 & ---    & 2.8812E-01 & --- \\
    4     & 6.4374E+01 & 1.69  & 5.5939E+00 & 1.62  & 6.7346E-02 & 2.10  \\
    8     & 1.8154E+01 & 1.83  & 1.0413E+00 & 2.43  & 5.3369E-03 & 3.66  \\
    12    & 8.3421E+00 & 1.92  & 3.9287E-01 & 2.40  & 1.1227E-03 & 3.84  \\
    16    & 4.7626E+00 & 1.95  & 1.9940E-01 & 2.36  & 3.6485E-04 & 3.91  \\
    20    & 3.0734E+00 & 1.96  & 1.1891E-01 & 2.32  & 1.5155E-04 & 3.94  \\
    24    & 2.1455E+00 & 1.97  & 7.8441E-02 & 2.28  & 7.3723E-05 & 3.95  \\
    28    & 1.5820E+00 & 1.98  & 5.5427E-02 & 2.25  & 4.0056E-05 & 3.96  \\
    32    & 1.2144E+00 & 1.98  & 4.1159E-02 & 2.23  & 2.3610E-05 & 3.96  \\
    \hline
    \end{tabular}%
  \label{ex_polyP3_2}%
\end{table}%

Next we consider the Crank-Nicolson scheme. When $k=2$, $j=k+3$ and $P=1000$, by Table \ref{ex_triP2_2}, we can observe the convergence rates on triangular meshes are $O(h)$ in the $H^2$ norm and $O(h^2)$ in the $L^2$ norm. Table \ref{ex_triP3_2} shows that when $k=3$, $j=k+4$ and $P=40000$, in triangular meshes, the convergence rates are respectively $O(h^2)$ and $O(h^4)$ in $H^2$ and $L^2$ norms. In the case of polygonal meshes, when $k=3$, $j=k+6$ and $P=50000$, we can get the results as shown in Table \ref{ex_polyP3_2}. And the convergence rates in tables are coincident with our theoretical results.

\begin{table}[htbp]
  \centering
  \caption{Error values and convergence rates for (\ref{ex}) on triangular meshes using $n=128$ and $\theta =1$}
    \begin{tabular}{c|c|c|c|c|c|c}
    \hline
    $P$ & $\trb{Q_hu -u _h} $ & Rate & $\norm{Q_h u-u_h}_{2,h}$& Rate & $\norm{Q_hu -u _h}$ & Rate \\
    \hline
    4     & 3.4948E-02 & ---    & 3.4948E-02 & ---    & 8.8374E-04 & --- \\
    8     & 3.2711E-02 & 0.10  & 3.2711E-02 & 0.10  & 8.2688E-04 & 0.10  \\
    16    & 2.0027E-02 & 0.71  & 2.0026E-02 & 0.71  & 5.0633E-04 & 0.71  \\
    32    & 1.0725E-02 & 0.90  & 1.0724E-02 & 0.90  & 2.7115E-04 & 0.90  \\
    64    & 5.5148E-03 & 0.96  & 5.5126E-03 & 0.96  & 1.3942E-04 & 0.96  \\
    128   & 2.7799E-03 & 0.99  & 2.7750E-03 & 0.99  & 7.0114E-05 & 0.99  \\
    \hline
    \end{tabular}%
  \label{ex_tri_time1}%
\end{table}%

\begin{table}[htbp]
  \centering
  \caption{Error values and convergence rates for (\ref{ex}) on polygonal meshes using $n=128$ and $\theta =1$}
    \begin{tabular}{c|c|c|c|c|c|c}
    \hline
    $P$ & $\trb{Q_hu -u _h} $ & Rate & $\norm{Q_h u-u_h}_{2,h}$& Rate & $\norm{Q_hu -u _h}$ & Rate \\
    \hline
    4     & 3.4466E-02 & ---    & 3.4463E-02 & ---    & 8.7254E-04 & --- \\
    8     & 3.2215E-02 & 0.10  & 3.2213E-02 & 0.10  & 8.1528E-04 & 0.10  \\
    16    & 1.9487E-02 & 0.73  & 1.9482E-02 & 0.73  & 4.9325E-04 & 0.72  \\
    32    & 1.0247E-02 & 0.93  & 1.0236E-02 & 0.93  & 2.5965E-04 & 0.93  \\
    64    & 5.0331E-03 & 1.03  & 5.0105E-03 & 1.03  & 1.2699E-04 & 1.03  \\
    128   & 2.3845E-03 & 1.08  & 2.3356E-03 & 1.10  & 5.8254E-05 & 1.12  \\
    \hline
    \end{tabular}%
  \label{ex_poly_time1}%
\end{table}%

\begin{table}[htbp]
  \centering
  \caption{Error values and convergence rates for (\ref{ex}) on triangular meshes using $n=28$ and $\theta =\frac{1}{2}$}
    \begin{tabular}{c|c|c|c|c|c|c}
    \hline
    $P$ & $\trb{Q_hu -u _h} $ & Rate & $\norm{Q_h u-u_h}_{2,h}$& Rate & $\norm{Q_hu -u _h}$ & Rate \\
    \hline
    4     & 7.2503E-02 & ---    & 7.2503E-02 & ---    & 1.8328E-03 & --- \\
    8     & 6.7287E-03 & 3.43  & 6.7287E-03 & 3.43  & 1.7079E-04 & 3.42  \\
    16    & 8.9315E-04 & 2.91  & 8.9316E-04 & 2.91  & 2.2726E-05 & 2.91  \\
    32    & 1.9070E-04 & 2.23  & 1.9070E-04 & 2.23  & 4.8551E-06 & 2.23  \\
    64    & 4.5478E-05 & 2.07  & 4.5478E-05 & 2.07  & 1.1557E-06 & 2.07  \\
    128   & 1.1434E-05 & 1.99  & 1.1435E-05 & 1.99  & 2.8416E-07 & 2.02  \\
    \hline
    \end{tabular}%
  \label{ex_tri_time2}%
\end{table}%

\begin{table}[htbp]
  \centering
  \caption{Error values and convergence rates for (\ref{ex}) on polygonal meshes using $n=28$ and $\theta =\frac{1}{2}$}
    \begin{tabular}{c|c|c|c|c|c|c}
    \hline
    $P$ & $\trb{Q_hu -u _h} $ & Rate & $\norm{Q_h u-u_h}_{2,h}$& Rate & $\norm{Q_hu -u _h}$ & Rate \\
    \hline
    4     & 7.2504E-02 & ---    & 7.2504E-02 & ---    & 1.8328E-03 & --- \\
    8     & 6.7277E-03 & 3.43  & 6.7277E-03 & 3.43  & 1.7075E-04 & 3.42  \\
    16    & 8.8912E-04 & 2.92  & 8.8912E-04 & 2.92  & 2.2604E-05 & 2.92  \\
    32    & 1.8469E-04 & 2.27  & 1.8469E-04 & 2.27  & 4.6670E-06 & 2.28  \\
    64    & 3.9010E-05 & 2.24  & 3.9008E-05 & 2.24  & 9.3371E-07 & 2.32  \\
    128   & 7.9670E-06 & 2.29  & 7.9673E-06 & 2.29  & 8.6452E-08 & 3.43  \\
    \hline
    \end{tabular}%
  \label{ex_poly_time2}%
\end{table}%

Finally, we verify the convergence rates about the time variable. When $\theta =1$, with $k=4$ and the space partition $n=128$, we get the order $O(h)$ whether it's triangulation using $j=k+3$ or polygon using $j=k+6$ from the Table \ref{ex_tri_time1} and Table \ref{ex_poly_time1}. Likewise, we test the case when $\theta =\frac{1}{2}$, $k=4$, $n=28$ and obtain the convergence rate $O(h^2)$ which are shown in Table \ref{ex_tri_time2} and Table \ref{ex_poly_time2}. The convergence rates coincide with the Theorem \ref{full-est}.

\section{Concluding remarks and ongoing work}
In this paper, we propose the semi-discrete numerical scheme and full-discrete numerical scheme for the
fourth order parabolic problem by the SFWG method and the implicit $\theta$-schemes where $\frac{1}{2}\leq\theta\leq 1$. Based on the schemes, we analyze the well-posedness and the convergence of the errors in $H^2$ and $L^2$ norms. Finally, we use numerical results to confirm the accuracy of the theoretical results.

In the future work, we are going to study other WG corresponding methods to deal with the fourth order parabolic equation and apply the WG methods combining the implicit $\theta$-schemes to solve other time-dependent partial differential equations.

\bibliographystyle{siam}
\bibliography{library}

\end{document}